\documentclass[11pt]{article}

\usepackage{amsmath, amssymb, amscd, amsthm, amsfonts}
\usepackage{graphicx}
\usepackage{hyperref}
\usepackage{indentfirst}
\usepackage{mathtools}
\usepackage{enumerate}
\usepackage{geometry}
\usepackage{gensymb}

\title{\textsc{\Large{\textbf{Revisiting Non-Rotating Star Models: Classical Existence and Uniqueness Theory and Scaling Relations}}}}

\author{Hangsheng Chen\thanks{Department of Mathematics, Statistics, and Computer Science, M/C 249, University of Illinois at Chicago, 851 S. Morgan Street, Chicago, IL 60607, USA. Email: hchen261@uic.edu}}

\date{ }

\newtheorem{theorem}{Theorem}

\newtheorem  {corollary} [theorem]{Corollary}

\newtheorem  {lemma} [theorem]{Lemma}
\newtheorem  {proposition}[theorem]{Proposition}

\theoremstyle{definition}

\newtheorem{assumption}[theorem]{Assumption}
\newtheorem  {remark} [theorem]{Remark}

\numberwithin{theorem}{section}

\makeatletter
\newcommand{\mylabel}[2]{#2\def\@currentlabel{#2}\label{#1}}
\makeatother

\begin{document}

\maketitle

\begin{abstract}\label{abstract}	
	This paper presents a systematic study of the properties of non-rotating stellar models governed by the Euler-Poisson system under general equations of state, including the case of polytropic gaseous stars. We revisit and extend existence results by Auchmuty and Beals \cite{AB71}, adapt the uniqueness results from the quantum mechanical framework of Lieb and Yau \cite{LY87} to the classical Newtonian mechanical setting. The results are also synthesized in McCann \cite{McC06} but without proof. The second work we do is applying a scaling method to establish relations between solutions with different total masses. As the mass tends to zero, we analyze convergence properties of the density functions and identify precise rates for the contraction or extension of their supports.
	
\end{abstract}

 \tableofcontents
 
\section{Introduction}\label{section-introduction}

\setlength{\parindent}{1.5em}

%

The motion of stars and planets has long been a classical subject in celestial mechanics, astrophysics, and mathematics. Historical models such as heliocentrism, proposed by Copernicus, and later studied by Newton and others on stellar rotation, form the foundation of rotating star or planet models. Early developments by Maclaurin, Jacobi, Poincaré, and Liapounov focused on incompressible, homogeneous or nearly homogeneous bodies like the Earth. For gaseous bodies such as Jupiter, stars, or systems like galaxies, density inhomogeneity and compressibility need to be accounted for. In these contexts, such bodies can be modeled as self-gravitating fluids, governed by the Euler–Poisson equations which have been extensively analyzed under various assumptions and classes of equations of state. We consider a fluid with density $\rho(x,t)\geq0$, velocity $v(x,t)\in\mathbb{R}^3$, and gravitational potential $-V(x,t)\in\mathbb{R}$, where $x\in\mathbb{R}^3$ is position and $t\geq0$ is time. Assuming the pressure $P(\rho)$ depends on the density only (with appropriate assumptions to be described later), the Euler–Poisson system is given by

\begin{equation} \label{EP}
	\begin{split}
		\partial_{t}\rho + \nabla \cdot \left( {\rho v} \right) &= 0,    \\
		\rho\partial_{t}v + \rho\left( {v \cdot \nabla} \right)v + \nabla P(\rho) &= \rho\nabla V,	\\
		\Delta V &= - 4\pi\rho.
	\end{split}	\tag{EP}
\end{equation}

The first attempts to construct axisymmetric rotating solutions to the compressible Euler–Poisson equations were made by astrophysicists Milne \cite{Mil23} and Zeipel \cite{Zei24} in the early 1920s, followed by Chandrasekhar \cite{Cha33} and Lichtenstein \cite{Lic33} in 1933. Further historical accounts can be found in \cite{Cha67, Jar13}.

Significant progress has been made for the single rotating star problem, where the star rotates about its axis. A special case of the rotating system is when the velocity $v$ is zero. In this case, the system is actually non-rotating, and the density $\rho$ is time-independent (stationary). Under these conditions, the Euler-Poisson equations reduce to the following equation:

\begin{equation} \label{EP'}
	\nabla P\left( \tilde{\rho}\left( x \right) \right) -\tilde{\rho}\left( x \right) \left( \nabla V_{\tilde{\rho}}\left( x \right) \right) =0
	\tag{EP'}
\end{equation}
where $\tilde{\rho}$ is a density function (in general with compact support), and $V_{\tilde{\rho}}(x) = {\int_{\mathbb{R}^3}{\frac{\tilde{\rho}(y)}{\left| {y - x} \right|}\,dy}}$.

For the non-rotating case, when the pressure satisfies a polytropic law, the solutions are known as Lane–Emden solutions. However, to obtain solutions for more general pressure laws and more general angular momentum (thus including the rotating case), two approaches have been employed. The first is variational method, as used by Auchmuty and Beals \cite{AB71M, AB71} to establish the existence of axisymmetric equilibria for compressible fluids; see also \cite{Auc91, AB71, CF80, CL94, Li91}. The second is a perturbative method based on the implicit function theorem, applied near non-rotating stars \cite{JM17, JM19, Hei94, SW17, SW19}. Nonlinear dynamical stability of rotating stars was shown in \cite{LS08, LS09} via variational techniques, while stability theory for non-rotating stars can be found in \cite{DLYY02, Jan14, Rei03}.

Each method has its advantages and limitations. For example, the variational approach typically requires $\gamma>\frac{4}{3}$ when $P(\rho)=K\rho^\gamma$ to prevent gravitational collapse (see Remark \ref{collapse} and Proposition \ref{unbounded below} of this paper, \cite[Section 2]{McC06}, or \cite[Section 1, 6 and 8]{AB71}). In contrast, perturbative methods can handle the range $\frac{6}{5}<\gamma<2$ \cite{JM19}, but they rely on prior knowledge of solutions in special non-rotating cases.

These methods are quite general and apply not only to single-body problems but also to multi-body systems. Variational methods were used by McCann \cite{McC06} for binary stars and star-planet systems in the authors' accompanying work \cite{Che26E3}. Perturbative approaches appear in the work of Alonso-Orán, Kepka, and Velázquez \cite{AKV23} on incompressible Euler–Poisson systems with an external particle, as well as in Lichtenstein’s construction of rotating binary stars \cite{Lic33}.

Researchers have also considered systems with more than two stellar objects, such as galaxies or general $N$-star systems. For rotating galaxies, single rotating galaxies were constructed in \cite{Sch08, SW17} in the spirit of Lichtenstein and Heilig \cite{Hei94, Lic33}. There is extensive literature on non-rotating galaxies: see \cite{BFH86, BT87, GR01, Rei07} and references therein. The orbital stability of stationary solutions has seen considerable progress over the past two decades \cite{Guo99, GL08, GR99E, GR99S, GR03, GR07, LMR05, LMR08, LMR11, LMR12, Rei07, Wol99}, with recent work on linearly oscillating galaxies in \cite{HRS22}. In \cite{CDD10}, Campos, Del Pino, and Dolbeault constructed $N$-body rotating galaxies by connecting them to relative equilibria (point particles) in $N$-body dynamics with small uniform angular velocity. They perturbed radial equilibria, and solved for the gravitational potential rather than the density, using a finite-dimensional reduction. Although natural, Jang and Seok remark in \cite{JS22} that the framework in \cite{CDD10} is not ideal for stability analysis. They adapt Rein’s reduction method \cite{Rei03, Rei07} to extend the discussion of binary star solutions to establish existence and stability for rotating binary galaxies modeled by the Vlasov–Poisson system \cite{JS22}. Furthermore, they note that for $N\geq 3$, uniformly rotating $N$-body stellar configurations lack a variational structure analogous to the binary case and are not generally expected to be stable \cite{JS22}. 

While recent research has extended to rotating multi-body systems, a thorough understanding of the simpler, non-rotating single-body problem remains very important. In addition to the existence result established via variational methods by Auchmuty and Beals \cite{AB71}, Lieb and Yau \cite{LY87} contributed results on uniqueness and other properties of ground states within the framework of quantum mechanics. The properties of these isolated bodies are often inspiring and provide technical tools required to handle the two-body case. These properties were hence later synthesized and extended by McCann \cite{McC06}. However, McCann did not provide detailed proofs, and some of the propositions are not entirely straightforward to justify.

%
\medskip
This paper aims to achieve the following two goals:
\begin{enumerate}[(a)]
	\item \label{goal (a)} To rigorously justify and extend the classical results. We revisit the existence and structure of non-rotating stars given by Auchmuty and Beals \cite{AB71} and McCann \cite{McC06} (Theorem \ref{non-rotating}). Following a suggestion of McCann \cite{McC06}, we prove uniqueness result in the classical Newtonian mechanical setting by adapting the quantum‑mechanical variational framework of Lieb and Yau \cite{LY87} to this setting (Theorem \ref{uniqueness in non-rotating cases}).
	\item \label{goal (b)} To explore the dependence of properties on the total mass of the system under polytropic law by employing a scaling method (Theorem \ref{scaling relation for energy} and Proposition \ref{scaling relation for derivatives}). The scaling analysis not only clarifies relations among solutions of different mass but also provides quantitative convergence results in the small-mass limit. In particular, we examine the decay rates of density profiles and the contraction or extension behavior of the supports of solutions (Remark \ref{asymptotic behaviour of radius and norm in single star case}).
\end{enumerate}

The results of this paper are adapted from the author’s Master’s thesis \cite{CVD24}, which also gave rise to two companion papers: \cite{Che26G1} and \cite{Che26E3}. The present work holds intrinsic interest within the theory of self-gravitating fluids. Furthermore, serving as a bridge within the thesis, this work builds upon conclusions from \cite{Che26G1} and, in turn, provides essential a priori estimates for the study of rotating star–planet systems in \cite{Che26E3}, where comparisons to the non-rotating case are essential. In \cite{Che26G1}, the conversion from the Euler–Lagrange equation to the Euler–Poisson equation was rigorously established, with particular attention given to showing that the energy-minimizing solutions still satisfy the Euler–Poisson equations at their boundaries. Nevertheless, this work is largely self-contained. Results from \cite{Che26G1} are given with proof outlines herein or cited as isolated technical details with complete proofs in \cite{Che26G1}.
%

The paper is organized according to the goals above. In Section \ref{section2-construction and results}, we introduce the mathematical settings and accomplish goal (\ref{goal (a)}), reviewing and extending the classical results. In Section \ref{section3-scaling relations}, we address goal (\ref{goal (b)}) through a scaling analysis.

\section{Construction and Results of Non-Rotating Systems}\label{section2-construction and results}

In this section we revisit and adapt the results theorems for non-rotating stars given by Auchmuty and Beals \cite{AB71} and Lieb and Yau \cite{LY87}, based on McCann's or the author's construction and assumptions \cite{McC06} \cite{Che26G1}. In the first subsection, we introduce the problem settings and some basic properties. In the second subsection, we review the existence theorems for non-rotating stars given by Auchmuty and Beals \cite{AB71} and McCann \cite{McC06}. In the third subsection, we adapt the uniqueness results from the quantum mechanical framework of Lieb and Yau \cite{LY87} to the classical Newtonian mechanical setting.

\subsection{Problem settings}\label{subsection2.1-notations and problem setting}



The state of a fluid may be represented by its mass density $\rho(x) \geq 0$ and velocity vector field $v(x)=0$ on $\mathbb{R}^3$ for non-rotating case. The fluid interacts with itself through Newtonian gravity hence we need to consider gravitational interaction energy, which will be given later. Moreover, to define internal energy, we consider a general form of the pressure $P(\rho)$ as the following (which is consistent with \cite{Che26G1}, see also \cite{AB71, McC06, JS22}):

\begin{assumption}
	\mbox{}
	\begin{enumerate}
		\item [\mylabel{F1}{(F1)}] $P:[0, \infty) \rightarrow[0, \infty)$ continuous and strictly increasing;
		\item [\mylabel{F2}{(F2)}] $\lim\limits_{s \rightarrow 0} P(s) s^{-\frac{4}{3}}=0$;
		\item [\mylabel{F3}{(F3)}] $\lim\limits_{s \rightarrow \infty} P(s) s^{-\frac{4}{3}}=\infty$.
	\end{enumerate}
\end{assumption}

Results under a weakened version of \ref{F3}, denoted as \ref{F3'}, will also be discussed later; see Remark \ref{collapse}. Here \textbf{assumption} \ref{F3'} is given as follows:
\begin{equation}\label{F3'}
	\lim\limits_{s \rightarrow \infty} \inf P(s) s^{-\frac{4}{3}}>K \tag*{(F3')}
\end{equation}
where $K>0$.

With these assumptions, we also define $A(s)$ as the following:
\begin{equation}\label{A}
	A(s):=\int_1^{\infty} P\left(\frac{s}{v}\right) d v=s \int_0^s P(\tau) \tau^{-2} d \tau
\end{equation}

Assuming that $\sigma$ is an energy minimizer, in order to show $\sigma$ is the solution to (\ref{EP'}) in the classical sense, we want the differentiability of $\sigma$ or of $P(\sigma)$, hence we introduce the following \textbf{assumption}: 
\begin{enumerate}
	\item[\mylabel{F4}{(F4)}] $P(\rho)$ is continuously differentiable on $[0, \infty)$, and $P^{\prime}(\rho)>0$ if $\rho>0$.
\end{enumerate}

Actually in \cite{Che26G1} we introduce another more general but also somewhat technical \textbf{assumption}:
\begin{enumerate}
	\item[\mylabel{F4'}{(F4')}] $P(\rho)$ is continuously differentiable on $[0, \infty)$. If $\rho>0, P(\rho)$ has non-vanishing (first order or higher order) derivative at $\rho$. That is, $\exists n \geq 1$, such that $P^{(n)}(\rho)$ exists and is not 0.
\end{enumerate}

Due to \ref{F2} we know $A(s)$ is well-defined. We also list some properties of $A(s)$ from another paper by the author \cite[Section 2]{Che26G1}. 

\begin{lemma}[Properties of $A$ {\cite[Section 2]{Che26G1}}]\label{properties of A}
	Let $A(s)$ is defined above (\ref{A}), then
	\begin{enumerate}[(i)]
		\item [\mylabel{i}{($i$)}] $A(s)$ is strictly increasing;
		\setcounter{enumi}{1}
		\item $A$ also satisfies \ref{F2} and \ref{F3} (or \ref{F3'} if we assume $P$ satisfies \ref{F3'} instead of \ref{F3}, though the constant $K$ can be different);
		\item $A^{\prime}$ satisfies: $\lim\limits_{s \rightarrow 0} A^{\prime}(s) \rho^{-\frac{1}{3}}=0$ and $\lim\limits_{s \rightarrow \infty} A^{\prime}(s) s^{-\frac{1}{3}}=\infty$;
		\item $A'(s)$ is continuous and 
		\begin{equation}\label{A'}
			A^{\prime}(s)=\left\{\begin{array}{r}\int_0^s P(t) t^{-2} d t+\frac{P(s)}{s}, s>0 \\ 0, s=0\end{array}\right.
		\end{equation}
		Moreover, $A^{\prime}(s) s-A(s)= P(s)$, $A^{\prime \prime}(s)=\frac{P^{\prime}(s)}{s} \geq 0$ a.e.;
		\item $A'$ is continuous and strictly increasing. Moreover, $A$ is convex;
		\item The inverse function of $A^{\prime}$, denoted by $\left(A^{\prime}\right)^{-1}$ or $\phi$, is well defined on $[0, \infty)$. Moreover, $\phi=\left(A^{\prime}\right)^{-1}$ is continuous since $A^{\prime}$ is continuous;
		\item If we assume $P(\rho)$ satisfies \ref{F4}, then $A^{\prime \prime}(\rho)$ exists and $A^{\prime \prime}(\rho)=\frac{P^{\prime}(\rho)}{\rho} \neq 0$ is continuous if $\rho>0$. In particular, $\phi=\left(A^{\prime}\right)^{-1} \in C^1((0, \infty))$.
	\end{enumerate}
\end{lemma}

If polytropic equations of state $P(s)=Ks^\gamma$ holds, where the parameter $\gamma > \frac{4}{3}$,  then easy to check $P$ satisfies \ref{F1} \ref{F2} \ref{F3}\ref{F4} or \ref{F1}\ref{F2}\ref{F3}\ref{F4'}, and
\begin{equation}\label{AwithP}
	A(s)=\frac{K}{\gamma-1} s^\gamma
\end{equation}
Hence one can check $A(s)$ indeed satisfies Lemma \ref{properties of A}.

In the following sections, we assume $P(\rho)$ satisfies \ref{F1}\ref{F2}\ref{F3} unless otherwise specified. We will mention \ref{F4} or \ref{F4'} or other assumptions when we want to use them.

We consider an ``admissible class'' for $\rho$ as the following:
\begin{align}\label{admissible class}
	R\left(\mathbb{R}^3\right) & :=\left\{\left.\rho \in L^{\frac{4}{3}}\left(\mathbb{R}^3\right) \right\rvert\, \rho \geq 0, \int_{\mathbb{R}^3} \rho \,dx=1\right\} 
\end{align}

Then given $\rho$ in such set, we consider energy $E_{0}(\rho)$ consisting of two terms:

\begin{equation}\label{energy}
	E_{0}(\rho):=U(\rho)-\frac{G(\rho, \rho)}{2}
\end{equation}

\begin{equation} \label{U}
	U(\rho):=\int_{\mathbb{R}^3} A(\rho(x)) \,dx
\end{equation}

\begin{equation} \label{G}
	G(\sigma, \rho):=\int_{\mathbb{R}^3} V_\sigma \rho\,dx=\iint_{{\mathbb{R}^3}\times {\mathbb{R}^3}} \frac{\rho(x) \sigma(y)}{|x-y|} \,dxdy
\end{equation}

Here $U(\rho)$ is the \textit{internal energy}, and $G(\rho, \rho)$ is the \textit{gravitational potential energy} (also called \textit{gravitational interaction energy}).


\begin{remark}
	In fact, Auchmuty and Beals \cite{AB71} consider a more general form of the energy with angular momentum
	\begin{equation} \label{energy in AB71}
	E(\rho)=E_0(\rho)+\frac{1}{2} \int_{\mathbb{R}^3} \rho(x) L\left(m_\rho(r(x))\right) r^{-2}(x) \,dx
	\end{equation}
	
	Another model for uniformly rotating stars was studied by McCann \cite{McC06}, where the total angular momentum $\boldsymbol{J}=J\hat{e}_z=(0,0,J)^T$ is prescribed. The corresponding energy $E_J(\rho)$ is defined as the following: 
	
	\begin{equation} \label{energy of UR}
		E_J(\rho)=U(\rho)-\frac{G(\rho, \rho)}{2}+T_J(\rho)
	\end{equation}
	with the \textit{kinetic energy} $T_J(\rho)$ given by
	\begin{equation}\label{T_J}
		T_J(\rho):=\frac{J^2}{2 I(\rho)}
	\end{equation}
	
	In this paper we mainly consider the non-rotating case ($L=0$ or $J=0$, reflected in the subscript ${0}$ in $E_{0}(\rho)$). This is not only for the sake of simplicity but also helps to produce more results such as uniqueness result in subsection \ref{subsection2.3-uniqueness results for non-rotating bodies}.
\end{remark}

\begin{remark}\label{finite gravitational interaction energy}
	Since $\rho \in L^{\frac{4}{3}}\left( \mathbb{R}^{3} \right) \cap L^{1}\left( \mathbb{R}^{3} \right)$, we have $G(\rho,\rho)<\infty$, see for example \cite[Section 2]{Che26G1}. Hence $E_{0}(\rho)$ is well-defined.
\end{remark}

Actually we can generalize our definition of $E_0(\rho)$ to the set 
\begin{equation}
	\begin{aligned}
		mR\left(\mathbb{R}^3\right) & :=\left\{\left.\rho \in L^{\frac{4}{3}}\left(\mathbb{R}^3\right) \right\rvert\, \rho \geq 0, \int_{\mathbb{R}^3} \rho\,dx=m\right\} \\
	\end{aligned}
\end{equation}

\begin{remark}
	From Remark \ref{finite gravitational interaction energy}, we see one reason why we require $\rho$ to belong not only to $L^1$ but also to $L^{\frac{4}{3}}$. Furthermore, selecting the specific exponent $\frac{4}{3}$ can also be shown as a natural assumption if we hope the minimal energy $\inf\limits_{\rho \in m{R}\left(\mathbb{R}^3\right)} E_0(\rho)$ to be finite. We will discuss this later (Proposition \ref{unbounded below}, see also \cite[Section 1,6 and 8]{AB71}, and \cite[Section 2]{McC06}). We will also show actually $\frac{4}{3}$ is the critical number such that the maximal energy is finite (Corollary \ref{unbounded above}).
\end{remark}

Consider the non-rotating minimizer $\sigma_m$ of $E_0(\rho)$ among configurations of mass $m \in [0,\infty)$, the corresponding minimum energy is finite due to the remark above. For the sake of convenience, we denote them by
\begin{equation}\label{e_0}
	e_0(m):=E_0\left(\sigma_m\right)=\inf _{\rho \in {R}\left(\mathbb{R}^3\right)} E_0(m \rho).
\end{equation}

We also denote $e_0(1)$ by $e_0$.

Since the energy is translation-invariant, we may just consider looking for an energy minimizer $\rho$ such that the center of mass $\bar{x}(\rho)$ is $0$, where 
\begin{equation}\label{centerofmass}
	\bar{x}(\rho):=\frac{\int_{\mathbb{R}^3} x \rho(x) \,dx}{\int_{\mathbb{R}^3} \rho(x) \,dx}
\end{equation}
Additionally, let support of $\rho$ be the smallest closed set carrying the full mass of $\rho$ (denoted by spt $\rho$). Intuitively we hope spt $\rho$  to be compact and simply-connected, as this aligns with the case of a single-star system.

\subsection{Existence Results for Non-rotating Bodies}\label{subsection2.2-Existence Results for Non-rotating Bodies}

Recall that in the calculus of variations, if an energy functional has a minimizer, then this minimizer must satisfy the Euler–Lagrange equation\footnote{When the minimization is carried out over a restricted class, the Euler–Lagrange equation may take the form of an inequality or involve taking the positive part; see, for example, \cite[Section 2, Appendix A]{Che26G1}.}. Moreover, it turns out that the Euler–Lagrange equation is equivalent, in a certain sense, to the Euler–Poisson equation (see, for example, \cite[Section 2]{Che26G1} or \cite{JS22}). Based on this observation, we have Auchmuty, Beals, Lieb and Yau's results \cite{AB71, LY87}. They not only show the existence of one-body and non-rotating solutions of Euler-Poisson equations, but also indicate some useful properties. McCann lists those results as the following theorem:

\begin{theorem}[Non-rotating Stars {\cite{AB71, LY87, McC06}}]\label{non-rotating}
	For $E_0(\rho)$ from (\ref{energy}), $e_0(m)$ from (\ref{e_0}) and $m \in [0,  \infty)$:
	\begin{enumerate}[(i)]
		\item $E_0(\rho)$ attains its minimum $e_0(m)$ among $\rho$ such that $\rho \in m{R}\left(\mathbb{R}^3\right)$;
		\item $e_0(m)$ decreases continuously from $e_0(0)=0$ and is strictly concave;
	\end{enumerate}
	There are bounds $R_0(m)$ and $C_0(m)$ on the radius and central density, such that any mass $m$ minimizer $\sigma_m$ of $E_0(\rho)$ satisfies.
	\begin{enumerate}[(i)]
		\setcounter{enumi}{2}
		\item $\sigma_m$ is spherically symmetric and radially decreasing after translation;
		\item $\left\|\sigma_m\right\|_{L^{\infty}} \leq C_0(m)$;
		\item spt $\sigma_m$ is contained in a ball of radius $R_0(m)$;
		\item $\sigma_m$ is continuous; moreover, if  $P(\rho)$ satisfies \ref{F4}, then $\phi \in C^1 ((0,\infty))$ and $\sigma_m\in C^1 (\{\sigma_m>0\})$; 
		\item $\sigma_m$ satisfies 
		\begin{equation} \label{EL}
			A^{\prime}(\sigma_m(x))=\left[V_{\sigma_m}(x)+\lambda_m\right]_{+} \tag{EL}
		\end{equation}
		 on all of $\mathbb{R}^3$ and a single Lagrange multiplier $\lambda_m\left\{\begin{array}{l}<0, m>0 \\ =0, m=0\end{array}\right.$. Here $[\cdot]_{+}$ is the nonnegative (positive) part function defined as $[\lambda]_{+}:=\max \{\lambda, 0\}$;
		\item the left and right derivatives of $e_0(m)$ bound $\lambda_m$: $e_0^{\prime}\left(m^{+}\right) \coloneq\lim\limits_{\widetilde{m}\rightarrow m^{+}}{e_{0}^{'}\left( \widetilde{m} \right)} \leq \lambda_m \leq e_0^{\prime}\left(m^{-}\right)\coloneq \lim\limits_{\widetilde{m}\rightarrow m^{-}}{e_{0}^{'}\left( \widetilde{m} \right)}$;
		\item if $P(\rho)$ satisfies assumption \ref{F4'}, then $\rho$  satisfies the reduced Euler-Poisson equations (\ref{EP'}).
	\end{enumerate}
\end{theorem}

\begin{remark}
	Theorem \ref{non-rotating} (iii), which is based on the rearrangement inequality, actually implies the support of $\rho$ is simply connected. This corresponds to a single-star system and thus satisfies our expectation. In contrast, we point out that in McCann’s binary star system \cite{McC06}, while he shows that the support of $\rho$ lies in two separate balls, it does not prove that the support consists of exactly two connected components. Whether it truly has exactly two connected components—and thus corresponds to a binary star system—remains an open question worthy of discussion.
\end{remark}

\begin{remark}
	Notice Auchmuty and Beals only concluded that (\ref{EP'}) is satisfied in the region where $\rho > 0$ \cite[Theorem A]{AB71}, whereas in Theorem \ref{non-rotating} (ix) we establish a stronger result—namely, that the equation \eqref{EP'} holds over the entire $\mathbb{R}^3$ space.
\end{remark}

Given that $\rho$ is an energy minimizer, to exploit this property better, we introduce the following perturbation set $P_{\infty}(\rho)$(the motivation behind its definition and related properties can be found in \cite[Section 4]{Che26G1}), which depends on $\rho$:
\begin{equation}\label{perturbation sets}
	P_{\infty}(\rho):=\bigcup_{R<\infty} P_R(\rho)
\end{equation}

Here, $P_R(\rho)$ is defined as:
$$P_R(\rho)=\left\{ \sigma \in L^{\infty}(\mathbb{R}^3)\mid \begin{array}{ll}
	\sigma(x)=0, & \text{where } x \text{ statisfies } \rho(x) >R \text{ or } |X|>R\\
	\sigma(x)\geq 0, & \text{where } x \text{ statisfies } \rho(x) <R^{-1}
\end{array} \right\}$$

One can see $P_{\infty}(\rho)$ is a convex cone. To apply variational method, we need to consider the derivative of $E_0$:
	\begin{lemma}[Differentiability of Energy $E_0(\rho)$ {\cite[Section 5]{Che26G1}}]\label{diff. of energy}
	Given $\rho \in m{R}^3$ with $U(\rho)<\infty$, we have $E_0(\rho)$ is differentiable at $\rho$ in the direction of $P_\infty (\rho)$. The derivative at $\rho$ is $E_0^{\prime}(\rho)$ in the sense that $\forall \sigma\in P_{\infty}(\rho)$, $E_0^{\prime}(\rho)(\sigma)=\int_{\mathbb{R}^3} E_0^{\prime}(\rho) \sigma\,dx$\footnote{To remain consistent with the notation in \cite[Section 4]{AB71}, we use $E_0^{\prime}$ as the symbol for both linear functional and function,  provided it does not cause confusion.}, where on the right hand side the function $E_0^{\prime}(\rho)$ is given by
	\begin{equation}\label{variational derivative}
		E_0^{\prime}(\rho)(x):=A^{\prime}(\rho(x))-V_\rho(x)
	\end{equation}
\end{lemma}

\begin{remark}\label{diff. at mini.}
	To apply Lemma \ref{diff. of energy}, we need to verify $U(\rho) <\infty $. It turns out if $\rho \in L^1 \cap L^\infty$, or $\rho$ minimizes $E_0(\rho)$ locally (under the topology induced by Wasserstein $L^\infty$ distance) or globally on $m{R}\left(\mathbb{R}^3\right)$, then $U(\rho)<\infty$, see \cite[Section 5]{Che26G1}.
\end{remark}

\begin{proof}[Proof of Theorem \ref{non-rotating}]
	{
		\rm
		Here we provide a proof outline, supplement some ideas and details not mentioned explicitly in \cite{AB71, LY87}. One can refer to \cite{AB71, LY87} and references therein to understand some step(s) omitted here.
		
		When $m=0$, we can see $\rho = 0$, hence the results (i) (iii-vii) are easy to check. The results (ii) and (viii) can be proved in the same way as in the case $m>0$.  Therefore, we assume $m>0$ in the following arguments unless otherwise specified.
		
		To prove the existence of a minimizer \footnote{It is natural to consider direct method \cite{Dal93}. However, so far, we only know $m R\left(\mathbb{R}^3\right)$ is bounded in $L^1\left(\mathbb{R}^3\right)$, which only implies weak star convergence. It is not enough to show the weak continuity of $G(\rho, \rho)$ and then the weak lower semi-continuity of $E_0(\rho)$ with respect to $\rho$ on $W_\infty \coloneq m {R}\left(\mathbb{R}^3\right)$, see Proposition \ref{compactness of convolution for q>3} or Proposition \ref{compactness of convolution for q>=1}. Hence, the direct method may not be valid on $W_\infty $. We need to construct a constrained admissible class $W_R$.}, we first construct a constrained admissible class: 
		
		\begin{equation}\label{W_R}
			W_R = \left\{ 
			\rho_R \in mR(\mathbb{R}^3)  \;\middle|\; 
			\begin{aligned}
				& \rho_R \geq 0, \, \rho_R \text{ is axially symmetric w.r.t. } z\text{-axis}, \\
				& \text{centered in the sense that: }\int_{z<0} \rho_R \, dx = \int_{z>0} \rho_R \, dx = \frac{m}{2}, \\
				& \|\rho_R\|_{L^\infty} \leq R, \, \text{spt}(\rho_R) \subset B_R(0), \\
				& U(\rho_R) < \infty, \, G(\rho_R, \rho_R) < \infty
			\end{aligned}
			\right\}
		\end{equation}
		Here $B_R(\cdot)$ is the open ball defined by $B_R(\cdot):=\left\{x \in \mathbb{R}^3\mid |x-\cdot| <R\right\}$. The axial symmetry and centered condition in \eqref{W_R} ensure that the center of mass of $\rho_R$ is $(0,0,0)^T$. Notice that $B_R(0)$ is a bounded set, by Proposition \ref{compactness of convolution for q>3} or Proposition \ref{compactness of convolution for q>=1}, we can show the weak continuity of $G(\rho, \rho)$ with respect to $\rho$ on $W_R$. Then one can apply direct method to show the existence of a minimizer $\sigma_R$ (depending on $m$) on $W_R$. In order to show such minimizer $\sigma_R$ is truly a minimizer on the whole admissible class $W_\infty$ for sufficiently large $R$, we need the following 4 lemmas.
		\begin{enumerate}[(1)]
			\item [\mylabel{1}{(1)}]  Show the $L^{\infty}\left(\mathbb{R}^3\right)$ bound of $\sigma_R$ is uniform, i.e. $\forall R>R_0,\left\|\sigma_R\right\|_{L^{\infty}\left(\mathbb{R}^3\right)} \leq k_1$. (Lemma \ref{uniform bound of norm in single star case})
			\item [\mylabel{2}{(2)}] Show $\forall R>\widetilde{R}_0$, $E_0\left(\sigma_R\right) \leq \alpha<0$, $\lambda_R \leq l<0$, where $\lambda_R$ is the corresponding Lagrange multiplier. (Lemma \ref{uniform bound for energy and multipliers})
			\item [\mylabel{3}{(3)}] Show $\forall R>R_1$, spt $\sigma_R$ is contained in $\overline{B_{R_1}(0)}$. (Lemma \ref{uniform bound on the sizes in single star case})
			\item [\mylabel{4}{(4)}] Take $R^{\prime}=\max \left\{k_1, R_1\right\}+1$, show $\sigma_{R^{\prime}}$ is a minimizer on $W_\infty=m {R}\left(\mathbb{R}^3\right)$. (Lemma \ref{existence of global minimizer in single star case})
		\end{enumerate}
		
		The proofs of the above four lemmas are postponed to later text. Once these lemmas are established, we obtain a minimizer which satisfies Parts (iv) (v). Therefore, \textbf{Part (i)} holds true with finite $e_0(m)$ by the second lemma. With the strategies similar to that employed in the proofs of the above lemmas, one can show \textbf{Part (iv)} and Part (v) actually hold true for any minimizer $\sigma_m$ (observe that the last term in (\ref{can be relaxed to half mass}) can be replaced by $\frac{m}{4 R}$ since $\int_{B_R(0)} \sigma_m\,dx \geq \frac{m}{2}$ for sufficiently large $R$). More precisely, the proof of \textbf{Part (v)} also relies on Part (vi) to show that outside the ball of radius $R_0 (m)$, $\sigma_m$ is not only almost everywhere zero but truly zero everywhere. Moreover, we have the fact that the Lagrange multiplier $\lambda_m$ is negative, which follows from the similar arguments in the proof of Step \ref{1} (Lemma \ref{uniform bound of norm in single star case}) showing $\lambda_R<0$, see also Remark \ref{negative energy and Lagrange multiplier for any minimizer}.	
		
		Now we have the existence of the minimizer $\sigma_m$. Since $\sigma_m$ is a minimizer on $m {R}\left(\mathbb{R}^3\right)$, the set of admissible perturbations here is exactly $P_{\infty}\left(\sigma_m\right)$. One can apply (generalized) Lagrange multiplier theorem (\cite[Section 2]{AB71} or \cite[Appendix A]{Che26G1}), together with the fact that there are sufficiently many functions in $P_{\infty}(\sigma_m)$, to obtain the following relations (see also \cite[Lemma 2]{AB71} or \cite[Section 5]{Che26G1}):
		\begin{equation}\label{ELi global chap2}
			E_0^{\prime}\left(\sigma_m\right) \geq \lambda_m \text { a.e., }    
		\end{equation}
		and
		\begin{equation}\label{EL global chap2}
			E_0^{\prime}\left(\sigma_m\right)=\lambda_m  \text { a.e. on } \{\sigma_m(x)>0\}.
		\end{equation}
		
		Since $E_0^{\prime}\left(\sigma_m\right)(x)=A^{\prime}\left(\sigma_m(x)\right)-V_{\sigma_m}(x)$, (\ref{EL global chap2}) implies
		\begin{equation}\label{bound of A' global}
			A^{\prime}\left(\sigma_m(x)\right) \leq V_{\sigma_m}(x)+\lambda_m   \text { a.e. on } \{\sigma_m(x)>0\}
		\end{equation}
		
		
		Due to Lemma \ref{properties of A} (iii), notice $\lambda_m<0$, we can obtain $\sigma_m{ }^{\frac{1}{3}} \leq c V_{\sigma_m}$ a.e. on the set where $\sigma_m \geq K$ for some $K$ large enough. As in \cite[Section 2]{Che26G1}, we can show the bound of $\|\sigma_{m}\|_{L^p}$ and $\left\|V_{\sigma_{m}}\right\|_{L^{\widetilde{p}}}$ with $p$ and $\widetilde{p}$ increasing alternately and reaching infinity in finite steps, and then $V_{\sigma_m}$ is continuously differentiable (bootstrap method). On the other hand (\ref{ELi global chap2})(\ref{EL global chap2}) implies
		\begin{equation}\label{ELA global chap2}
			\sigma_m=\phi \circ\left[V_{\sigma_m}(x)+\lambda_m\right]_{+} \text { a.e. in } \mathbb{R}^3
		\end{equation}	
		where $\phi=\left(A^{\prime}\right)^{-1}$ is continuous. This means $\sigma_m$ has a continuous representative. Moreover, if $P(\rho)$ satisfies \ref{F4}, then $\phi \in C^1((0, \infty))$ by Lemma \ref{properties of A} (vii), which implies $\sigma_m \in C^1\left(\left\{\sigma_m>0\right\}\right)$ as shown in \cite[Section 2]{Che26G1}. (\textbf{Part (vi)}).

		Part (iii) comes from a strong rearrangement inequality in Lieb \cite[Lemma 3]{Lie77} and the generalization by Sobolev of the Hardy-Littlewood theorem on rearrangements of functions \cite{Sob38}. One can also refer to \cite[Proof of Theorem 6.1]{McC06} or \cite[Section 6]{Che26E3}. (\textbf{Part (iii)})
		
		For Part (vii), we already know $\lambda_m<0$. Moreover, (\ref{ELi global chap2}) (\ref{EL global chap2}), together with $\sigma_m$ is continuous, give (\ref{EL}), see for example \cite[Section 5]{McC06} or \cite[Section 2]{Che26G1}. (\textbf{Part (vii)})

		For Part (ii), we only consider the case that complies with polytropic equations of state $P(\sigma)=K\sigma^\gamma$. The general case can be found in  \cite[Theorem 3 (d)]{LY87}\footnote{However, it should be noted that the argument there is set within a quantum mechanical framework and does not include an internal energy term, although McCann mentioned that the results can be applied ``equally well to all $A(\rho)$ consistent with \ref{F1}-\ref{F3}" \cite[Section 3]{McC06}.}. We utilize the scaling relations mentioned in Section \ref{section3-scaling relations}. That is, given $\sigma$ a non-rotating minimizer with mass $1$, then $\sigma_m(x)=\frac{1}{A} \sigma\left(\frac{1}{B} x\right)$ will be a non-rotating minimizer with mass $m$, where $A(m)=m^{-\frac{2}{3 \gamma-4}}$, $B(m)=m^{\frac{\gamma-2}{3 \gamma-4}}$. At the same time, we know $e_0(m)=E_0\left(\sigma_m\right)=E_0\left(\frac{1}{A} \sigma\left(\frac{1}{B} x\right)\right)=m^{\frac{5 \gamma-6}{3 \gamma-4}} e_0$. Notice when $m=0$, such expression also holds since $e_0(0)=0$. By the similar arguments as the proof of Step \ref{2} (Lemma \ref{uniform bound for energy and multipliers}), we know $e_0 \coloneq e_0(1)<0$, which implies $e_0(m)$ decreases continuously from $e_0(0)=0$ and is strictly concave (\textbf{Part (ii)}).
		
		For Part (viii), first the concavity of $e_0(m)$ implies the existence of left-hand and right-hand derivations (since they are monotonic), and they are equal a.e. (since concavity implies Lipschitz continuity and then apply Rademacher's theorem \cite[Subsection 5.8.3]{Eva10}). Thanks to Euler-Lagrange equations (\ref{EL}), we multiply $\sigma_m$ on both sides and integrate, and then we obtain 
		$$\int_{\mathbb{R}^3} E_0^{\prime}\left(\sigma_m\right) \sigma_m\, dx=\lambda_m \cdot m$$ Notice due to Part (iv) and (v), we know $\sigma_m \in P_{\infty}\left(\sigma_m\right)$ thus $\int_{\mathbb{R}^3} E_0^{\prime}\left(\sigma_m\right) \sigma_m\, dx$ makes sense and $$\int_{\mathbb{R}^3} E_0^{\prime}\left(\sigma_m\right) \sigma_m\, dx=\lim\limits_{t \rightarrow 0^{+}} \frac{E_0\left(\sigma_m+t \sigma_m\right)-E_0\left(\sigma_m\right)}{t}=\lim\limits_{t \rightarrow 0^{-}} \frac{E_0\left(\sigma_m\right)-E_0\left(\sigma_m-t \sigma_m\right)}{t}$$ Since $\sigma_m+t \sigma_m$ has mass $(1+t) m$, then 
		$$e_0((1+t) m)=E_0\left(\sigma_{(1+t) m}\right) \leq E_0\left(\sigma_m+t \sigma_m\right)$$ Thus $$\lim\limits_{t \rightarrow 0^{+}} \frac{E_0\left(\sigma_m+t \sigma_m\right)-E_0\left(\sigma_m\right)}{t} \geq \lim \limits_{t \rightarrow 0^{+}} \frac{e_0((1+t) m)-e_0(m)}{t}=m e_0^{\prime}\left(m^{+}\right)$$ 
		Hence we have $\lambda_m \geq e_0^{\prime}\left(m^{+}\right)$. Similarly, we have $\lambda_m \leq e_0^{\prime}\left(m^{-}\right)$. (\textbf{Part (viii)})
		
		For Part (ix), thanks to Part (vi), we know the global continuity of $\sigma_m$, and then we can apply the same arguments as in \cite[Section 2]{Che26G1} and show the existence of $\nabla P\left(\sigma_m\right)$ on the boundary $\partial\left\{\sigma_m>0\right\}$, and verify $\sigma_m$ is a solution to reduced Euler-Poisson equations (\ref{EP'}) in $\mathbb{R}^3$. (\textbf{Part (ix)})
	}
\end{proof}

\begin{remark}
	In the proof of Part (iv) of Theorem \ref{non-rotating}, the relaxation of the last term in (\ref{can be relaxed to half mass}) to $\frac{m}{4 R}$ is not fully justified in \cite{AB71}; here we make this detail explicit for completeness.
\end{remark}

\begin{remark}\label{density less than potential}
	When we prove Theorem \ref{non-rotating} (vi), we mention $\sigma_m{ }^{\frac{1}{3}} \leq c V_{\sigma_m}$ a.e. on the set where $\sigma_m \geq K$ for some $K$ large enough. Actually it turns out even if we do not know  $\lambda_m<0$, $\sigma_m{ }^{\frac{1}{3}} \leq c V_{\sigma_m}$ can still hold true since we can absorb $\lambda_m$ when $\sigma_m$ is large enough. See \cite[Lemma 3]{AB71}.
\end{remark}

\begin{remark}\label{collapse}
	We point out in Theorem \ref{non-rotating} we can show the existence of non-rotating single star no matter how large the mass is. However, if we replace \ref{F3} by \ref{F3'}, then we can only obtain minimizers with total mass not larger than $m(K)$ via variational method for some $m(K)>0$. $m(K)$ here turns out to be the Chandrasekhar mass for the model. If the mass is too large, gravitational collapse may happen \cite[Appendix]{AB71}. 
	
	To clarify, replacing \ref{F3} with \ref{F3'} would require more careful treatment to prove Lemma \ref{non-rotating}. First, observe that Lemma \ref{properties of A} needs to be modified under assumption \ref{F3'} rather than \ref{F3}. In the proof of Lemma \ref{non-rotating}, a potential issue with a large mass is that it may invalidate the earlier assertion that ``$\sigma_m{ }^{\frac{1}{3}} \leq c V_{\sigma_m}$ a.e. on the set where $\sigma_m \geq K$ for some $K$ large enough". Moreover, \ref{F3'} may also lead to a more delicate treatment in the proof of Lemma \ref{uniform bound of norm in single star case} below; see also Remark \ref{F3 importance}. For more discussion, one can see \cite[Section 2]{McC06} or \cite[Section 1,6,8 and Appendix]{AB71}.
	
\end{remark}

\begin{remark}\label{uniform negative bound for derivative of energy}
	Theorem \ref{non-rotating} (ii) actually implies that $e^\prime_0(m^-) <0$ for $m>0$. To prove it, first note $e_0(m)<0$ for $m>0$. Otherwise the monotonicity would imply that $e_0=0$ on $[0, m]$, which contradicts the strict concavity property. Then due to strict concavity, $\forall m>0$, $\forall \epsilon>0$, $\frac{e(m+\epsilon)-e(m)}{\epsilon}<\frac{e(m)-e(0)}{m}=\frac{e(m)}{m}<0$, which implies $e^\prime_0(m^-) \leq \frac{e(m)}{m}<0$.
\end{remark}

\begin{remark}\label{motivaton for uniqueness 1}
	
	Currently, we do not know whether the minimizer $\sigma_m$ with fixed mass $m$ is unique. Further discussion regarding uniqueness (up to translation) will be conducted under additional assumptions later. However, even if multiple minimizers $\sigma_m^1, \sigma_m^2 \ldots \sigma_m^i \ldots$ exist under the present assumptions, due to Theorem \ref{non-rotating} (iv), together with interpolation inequality \cite[Section 4.2]{Bre11}, we know that all of them have uniform $L^p$ bound. Moreover, according to Theorem \ref{non-rotating} (viii) and Remark \ref{uniform negative bound for derivative of energy}, we observe their corresponding Lagrange multipliers $\lambda_m^i$ are also uniformly bounded by $e^\prime_0(m^-)$, which is strictly negative. This observation is applied in \cite[Section 6]{McC06}.
	
\end{remark}

Before showing the existence of a minimizer on $m {R}\left(\mathbb{R}^3\right)$, it is worth noting that if we replace $\frac{4}{3}$ by any smaller number in \ref{F3}, there is no minimizer since the energy may be unbounded below. The idea comes from \cite[Remark in Section 8]{AB71}.

\begin{proposition}[Unbounded below if $\gamma$ <$\frac{4}{3}$] \label{unbounded below}
	If $P(\rho)$ satisfies $\lim\limits_{\rho \rightarrow \infty} \sup P(\rho) \rho^{-\gamma}=K$ instead of \ref{F3}, where $\gamma<\frac{4}{3}, 0 \leq K<\infty$, then the energy $E_0(\rho)$ is unbounded below, that is $\inf\limits_{\rho \in m {R}\left(\mathbb{R}^3\right)} E_0(\rho)=-\infty$. In particular, there is no minimizer on $m {R}\left(\mathbb{R}^3\right)$.
\end{proposition}

\begin{proof}
	{
		\rm
		Follow the proof of L'Hôpital's rule \cite[Theorem 5.13]{Rud76}, we know $\lim\limits_{\rho \rightarrow \infty} \sup P(\rho) \rho^{-\gamma}=K$ implies $\lim\limits_{\rho \rightarrow \infty} \sup A(\rho) \rho^{-\gamma} \leq 3 K$. then we know $\lim \limits_{\rho \rightarrow \infty} A(\rho) \rho^{-\frac{4}{3}}=$ 0. Therefore, $\forall \hat{\epsilon}>0, \exists S_{\hat{\epsilon}}>0$, such that $A(s) \leq \hat{\epsilon} s^{-\frac{4}{3}}$ when $s \geq S_{\hat{\epsilon}}$. In the statement of Proposition \ref{bound of potential energy}, let $\hat{c}$ be the smallest constant such that $\forall \rho \in L^1 \cap L^{\frac{4}{3}}$, we have 
		$$\left|\int_{\mathbb{R}^3} \rho V_\rho\, dx\right| \leq 2 \hat{c}\left(\int_{\mathbb{R}^3}|\rho|\, dx\right)^{\frac{2}{3}} \int_{\mathbb{R}^3}|\rho|^{\frac{4}{3}}\, dx$$ 
		Then $\hat{c}>0$. Otherwise $\left|\int_{\mathbb{R}^3} \rho V_\rho\, dx\right| \leq 0$ implies $\rho=0$, but $0$ is not the only element in $L^1 \cap L^{\frac{4}{3}}$.
		
		Therefore, $\forall \epsilon \in(0,2 \hat{c})$, $\exists \widetilde{\rho_\epsilon} \in L^1 \cap L^{\frac{4}{3}}$, such that 
		$$\left|\int_{\mathbb{R}^3} \widetilde{\rho_\epsilon} V_{\widetilde{\rho_\epsilon}}\, dx\right|>2\left(\hat{c}-\frac{\epsilon}{2}\right)\left(\int_{\mathbb{R}^3}\left|\widetilde{\rho_\epsilon}\right|\,dx\right)^{\frac{2}{3}}\int_{\mathbb{R}^3}\left|\widetilde{\rho_\epsilon}\right|^{\frac{4}{3}}\, dx$$
		We can assume $\widetilde{\rho_\epsilon} \geq 0$ since we can replace $\widetilde{\rho_\epsilon}$ by $\left|\widetilde{\rho_\epsilon}\right|$ if needed. Since $C_c^0\left(\mathbb{R}^3\right)$ is dense in $L^1$ and $L^{\frac{4}{3}}$, where $C_c^0\left(\mathbb{R}^3\right)$ denotes the set of continuous functions with compact support, by using HardyLittlewood-Sobolev Inequality (Proposition \ref{HLSI}) or Proposition \ref{bound of potential}, we can find a $\rho_\epsilon \in C_c^0\left(\mathbb{R}^3\right)$ such that $\left|\int_{\mathbb{R}^3} \rho_\epsilon V_{\rho_\epsilon}\, dx-\int_{\mathbb{R}^3} \widetilde{\rho_\epsilon} V_{\widetilde{\rho_\epsilon}}\, dx\right|$ can be arbitrarily small. Furthermore, we can have
		$$
		\frac{\left|\int_{\mathbb{R}^3} \rho_\epsilon V_{\rho_\epsilon}\, dx\right|}{\left(\int_{\mathbb{R}^3}\left|\rho_\epsilon\right|\, dx\right)^{\frac{2}{3}} \int_{\mathbb{R}^3}\left|\rho_\epsilon\right|^{\frac{4}{3}}\, dx}>2(\hat{c}-\epsilon)
		$$
		We can also write it as 
		$$
		G\left(\rho_\epsilon, \rho_\epsilon\right) = \left|\int_{\mathbb{R}^3} \rho_\epsilon V_{\rho_\epsilon}\,dx\right|>2(\hat{c}-\epsilon)\left(\int_{\mathbb{R}^3}\left|\rho_\epsilon\right|\, dx\right)^{\frac{2}{3}} \int_{\mathbb{R}^3}\left|\rho_\epsilon\right|^{\frac{4}{3}}\, dx
		$$
		 It implies $\rho_\epsilon \neq 0$. We can also assume $\rho_\epsilon \geq 0$ since we can replace $\rho_\epsilon$ by $\left[\rho_\epsilon\right]_{+}$ if needed. Without loss of generality, we assume $\int_{\mathbb{R}^3} \rho_\epsilon\, dx=1$. Let $\rho_\delta(x)=m \delta^{-3} \rho_\epsilon\left(\delta^{-1} x\right)$, $\delta>0$. Then $\rho_\delta \in m {R}\left(\mathbb{R}^3\right)$ since $C_c^0\left(\mathbb{R}^3\right) \subset L^{\frac{4}{3}}$. Similar as discussions in Section \ref{section3-scaling relations}, we can show $G\left(\rho_\delta, \rho_\delta\right)=\frac{m^2}{\delta} G\left(\rho_\epsilon, \rho_\epsilon\right)$, and
		$$
		\begin{aligned}
			U\left(\rho_\delta\right)&=\int_{\mathbb{R}^3} A\left(\rho_\delta\right)\, dx\\
			&=\int_{\left\{\rho_\delta \geq S_{\hat{\epsilon}}\right\}} A\left(\rho_\delta\right)\, dx+\int_{\left\{\rho_\delta<S_{\hat{\epsilon}\}}\right.} A\left(\rho_\delta\right) \, dx\\
			&\leq \int_{\left\{\rho_\delta \geq S_{\hat{\epsilon}}\right\}} \hat{\epsilon} \cdot \rho_\delta^{\frac{4}{3}}\, dx+\int_{\left\{\rho_\delta<S_{\hat{\epsilon}\}}\right.} A\left(S_{\hat{\epsilon}}\right)\, dx \\
			&\leq \frac{m^{\frac{4}{3}} \hat{\epsilon}}{\delta} \int_{\mathbb{R}^3} \rho_\epsilon^{\frac{4}{3}}\, dx+A\left(S_{\hat{\epsilon}}\right) \mu\left(\left\{\rho_\delta<S_{\hat{\epsilon}}\right\}\right)
		\end{aligned}
		$$
		Here $\mu$ denotes the Lebesgue measure. Notice that given $\hat{\epsilon}$ and $\epsilon$, as $\delta \rightarrow 0$, we have $$\mu\left(\left\{\rho_\delta<S_{\hat{\epsilon}}\right\}\right)=\mu\left(\left\{x \left\lvert\, \rho_\epsilon\left(\delta^{-1} x\right)<\frac{S_{\hat{\epsilon}} \delta^3}{m}\right.\right\}\right)=\delta^3 \mu\left(y \left\lvert\,\left\{\rho_\epsilon(y)<\frac{S_{\hat{\epsilon}} \delta^3}{m}\right\}\right.\right) \rightarrow 0$$
		Then we pick $\epsilon=\frac{\hat{c}}{2}, \hat{\epsilon}=\frac{m^{\frac{2}{3}} \hat{c}}{4}$, and then when $\delta \rightarrow 0$
		$$
		\begin{aligned}
			E_0\left(\rho_\delta\right)&=U\left(\rho_\delta\right)-\frac{G\left(\rho_\delta, \rho_\delta\right)}{2} \\
			&\leq \frac{m^2 \hat{c}}{4 \delta} \int_{\mathbb{R}^3} \rho_\epsilon^{\frac{4}{3}}\, dx+A\left(S_{\hat{\epsilon}}\right) \mu\left(\left\{\rho_\delta<S_{\hat{\epsilon}}\right\}\right)-\frac{m^2 \hat{c}}{2 \delta} \int_{\mathbb{R}^3} \rho_\epsilon^{\frac{4}{3}} \, dx\rightarrow-\infty
		\end{aligned}
		$$
		
		Therefore, $\inf\limits_{\rho \in m {R}\left(\mathbb{R}^3\right)} E_0(\rho)=-\infty$.
	}
\end{proof}

\begin{remark}\label{unboundedness}
	Actually, there is another way to see the unboundedness. From $\lim\limits_{\rho \rightarrow \infty} \sup A(\rho) \rho^{-\gamma} \leq 3 K$, we can notice $A(\rho) \leq(3 K+\hat{\epsilon}) \rho^\gamma$ for large $\rho$, then similarly as above we have 
	$$
	\begin{aligned}
		E_0\left(\rho_\delta\right)&=U\left(\rho_\delta\right)-\frac{G\left(\rho_\delta, \rho_\delta\right)}{2} \\
		&\leq m^\gamma \delta^{-3+3 \gamma}(3 K+\hat{\epsilon}) \int_{\mathbb{R}^3} \rho_\epsilon^\gamma\, dx+A\left(S_{\hat{\epsilon}}\right) \mu\left(\left\{\rho_\delta<S_{\hat{\epsilon}}\right\}\right)-\frac{(\hat{c}-\epsilon) m^2}{\delta} \int_{\mathbb{R}^3}\left|\rho_\epsilon\right|^{\frac{4}{3}}\, dx
	\end{aligned}	
	$$
	Notice $-3+3 \gamma<-1$, thus we can choose suitable $\epsilon$ and $\hat{\epsilon}$ such that $E_0\left(\rho_\delta\right) \rightarrow-\infty$ as $\delta \rightarrow 0$. From this point we can know that the number $\gamma=\frac{4}{3}$ in \ref{F3} is chosen such that $-3+3 \gamma \geq-1$, then \ref{F3} can help the blowing up of $U(\rho)$ to control the gravitational collapse $G(\rho, \rho)$.
\end{remark}

In fact, in the case that blowing up of $U(\rho)$ is too large, with the same strategy we can show the energy may also be unbounded above. That is, if $\frac{4}{3}$ is replaced by any larger number in \ref{F3}, then we will have $U\left(\rho_\delta\right) \gg G\left(\rho_\delta, \rho_\delta\right) \gg 1$.

\begin{corollary}[Unbounded above if $\gamma>\frac{4}{3}$]\label{unbounded above}
	If $P(\rho)$ satisfies $\lim\limits_{\rho \rightarrow \infty} \inf P(\rho) \rho^{-\gamma}=K$ instead of \ref{F3}, where $\gamma>\frac{4}{3}$, $0<K \leq \infty$, then the energy $E_0(\rho)$ is unbounded above, that is $\sup\limits_{\rho \in m {R}\left(\mathbb{R}^3\right)} E_0(\rho)=+\infty$.
\end{corollary}

\begin{proof}
	{
		\rm
		Thanks to the arguments above, we can first find a $\rho \in C_c^0\left(\mathbb{R}^3\right)$ with mass 1, and set $\rho_\delta(x)=m \delta^{-3} \rho\left(\delta^{-1} x\right)$. If $K<\infty$, notice $$\left\{x \, \middle| \, \rho_\delta \geq S_{\hat{\epsilon}}\right\}=\left\{x \, \middle| \, \rho\left(\delta^{-1} x\right) \geq \frac{S_{\hat{e}} \delta^3}{m}\right\}$$ 
		by a change of variable, we know for $\delta$ large enough, we have 
		$$U\left(\rho_\delta\right) \geq m^\gamma \delta^{-3+3 \gamma}(3 K-\hat{\epsilon}) \int_{\left\{\rho \geq \frac{S_{\hat{\epsilon} } \delta^3}{m}\right\}} \rho^\gamma\, dx \geq m^\gamma \delta^{-3+3 \gamma}(3 K-\hat{\epsilon}) \int_{\{\rho \geq C\}} \rho^\gamma\, dx$$ Notice we can find a $C$ such that $\int_{\{\rho \geq C\}} \rho^\gamma\, dx>0$ due to the same arguments in \cite[Section 4]{Che26G1}. By Proposition \ref{bound of potential energy}, together with $\gamma>\frac{4}{3}$, we know as $\delta \rightarrow 0$, $$E_0\left(\rho_\delta\right)=U\left(\rho_\delta\right)-\frac{G\left(\rho_\delta, \rho_\delta\right)}{2} \geq m^\gamma \delta^{-3+3 \gamma}(3 K-\hat{\epsilon}) \int_{\{\rho \geq C\}} \rho^\gamma\, dx-\frac{\hat{c} m^2}{\delta} \int_{\mathbb{R}^3}|\rho|^{\frac{4}{3}}\, dx \rightarrow+\infty$$
		When $K=\infty$, we can replace $(3 K-\hat{\epsilon})$ by any number $N$ and $S_{\hat{\epsilon}}$ by $S_N$ analogously, where $S_N$ is a number such that $A(\rho) \geq N \rho^\gamma$ when $\rho>S_N$, then we can obtain the same result.
	}
\end{proof}

\begin{remark}
	In fact, the energy may be also unbounded below or above even if we replace ${R}\left(\mathbb{R}^3\right)$ by ${R}_p\left(\mathbb{R}^3\right):=\left\{\rho \in L^p\left(\mathbb{R}^3\right) \mid \rho \geq 0, \int_{\mathbb{R}^3} \rho\, dx=1\right\}$ for some $p\geq 1$, since $C_c^0\left(\mathbb{R}^3\right) \subset L^p$ for any $p \geq 1$, and then $\rho_\epsilon \in {R}_p\left(\mathbb{R}^3\right)$.
\end{remark}

\begin{remark}
	If \ref{F3} is replaced by the assumption $\lim\limits_{\rho \rightarrow \infty} \inf P(\rho) \rho^{-\frac{4}{3}}=K>0$, Auchmuty and Beals \cite[Section 1 and 6]{AB71} show there is a constant $M_0>0$ such that $\forall 0 \leq m<M_0$ there is a minimizer on $m {R}\left(\mathbb{R}^3\right)$. If $\lim\limits_{\rho \rightarrow \infty} \inf P(\rho) \rho^{-\frac{4}{3}}=K>0$ is strengthened to $\lim\limits_{\rho \rightarrow \infty} P(\rho) \rho^{-\frac{4}{3}}=K>0$, then $ \forall m>M_0$, the energy is unbounded below on $m {R}\left(\mathbb{R}^3\right)$ \cite[Remark in Section 8]{AB71}. Moreover, they \cite[Remark in Section 6, and Appendix]{AB71} also show that if $P$ is the function occurring in the theory of white dwarf stars (relativistically degenerate gases), $M_0$ turns out to be the same as the limiting mass $M_3$ in Chandrasekhar's theory \cite{Cha39}.
\end{remark}

We now prove the four steps (lemmas) that are mentioned during the earlier proof of Theorem \ref{non-rotating}. We assume $m>0$, and just provide proof outlines, explain some ideas and add some details not mentioned explicitly in \cite{AB71, LY87}. One can check the proofs in \cite{AB71} \cite{LY87} and references therein, to understand some step(s) omitted in the following.

\begin{lemma}[Uniform bound of $\sigma_R$ in $L^{\infty}\left(\mathbb{R}^3\right)$ \text{\cite{AB71}}] \label{uniform bound of norm in single star case}
	Assume $\sigma_R$ is the energy minimizer in $W_R$, which is given in (\ref{W_R}). Then the $L^{\infty}\left(\mathbb{R}^3\right)$ bound of $\sigma_R$ is uniform, i.e., $\exists C>0$, such that $\forall R>0$, $\left\|\sigma_R\right\|_{L^{\infty}\left(\mathbb{R}^3\right)} \leq C$.
\end{lemma}

\begin{remark}
	Here we assume that the inequality $\left\|\sigma_{R}\right\|_{L^{\infty}\left(\mathbb{R}^{3}\right)} \leq C$ holds automatically if $W_{R}$ is empty. But we will also show $W_R$ is not empty when $R$ is large enough in the following proof.
\end{remark}

\begin{proof}[Proof of Lemma \ref{uniform bound of norm in single star case}]
	{
		\rm
		Let $R_0=\left(\frac{3 m}{4 \pi}\right)^{\frac{1}{4}}$, then 
		$W_R \subset L^1 \cap L^{\infty}$ and it is not empty when $R>R_0$, since we can pick $\widetilde{\rho}(x)=\left\{\begin{array}{ll}\frac{3 m}{4 \pi R_0^3},&|x|<R_0 \\ 0,&|x| \geq R_0\end{array}\right.$ and check $\widetilde{\rho}$ is in $W_R$. Given $R>R_0$, we have $E_0\left(\sigma_R\right) \leq E_0(\widetilde{\rho})<\infty$. Thanks to Proposition \ref{bound of potential energy}, we have 
		$$\int_{\mathbb{R}^3} A\left(\sigma_R\right)\, dx \leq E_0(\widetilde{\rho})+C m^{\frac{2}{3}} \int_{\mathbb{R}^3} \sigma_R^{\frac{4}{3}}\, dx$$ 
		where $m=\int_{\mathbb{R}^3} \sigma_R\, dx$. Let $\widetilde{C}=C m^{\frac{2}{3}}$, by Lemma \ref{properties of A} (ii), $\exists s_0>0$, such that $\forall s \geq s_0$, $(\widetilde{C}+2) s^{\frac{4}{3}} \leq A(s)$. Notice the mass of $\sigma_R$ is $m$. Then
		$$
		\begin{aligned}
			\int_{\mathbb{R}^3}\left(\sigma_R\right)^{\frac{4}{3}}\, dx&=\int_{\left\{x \mid \sigma_R(x)<s_0\right\}}\left(\sigma_R\right)^{\frac{1}{3}} \cdot \sigma_R\, dx+\int_{\left\{x \mid \sigma_R(x) \geq s_0\right\}}\left(\sigma_R\right)^{\frac{4}{3}}\, dx \\
			& \leq s_0^{\frac{1}{3}} \cdot m+\frac{1}{(\widetilde{C}+2)} \int_{\mathbb{R}^3} A\left(\sigma_R\right)\, dx
		\end{aligned}
		$$

		Then claim: $\int_{\mathbb{R}^3}\left(\sigma_R\right)^{\frac{4}{3}} \, dx\leq \max \left\{(\widetilde{C}+2) s_0^{\frac{1}{3}} \cdot m, E_0(\widetilde{\rho})\right\}$. In fact, if $\int_{\mathbb{R}^3}\left(\sigma_R\right)^{\frac{4}{3}}\, dx \leq (\widetilde{C}+2) s_0^{\frac{1}{3}} \cdot m$, we are done. If $\int_{\mathbb{R}^3}\left(\sigma_R\right)^{\frac{4}{3}}\, dx>(\widetilde{C}+2) s_0^{\frac{1}{3}} \cdot m$, then 
		$$
		(\widetilde{C}+1) \int_{\mathbb{R}^3}\left(\sigma_R\right)^{\frac{4}{3}}\,dx \leq(\widetilde{C}+2) \int_{\mathbb{R}^3}\left(\sigma_R\right)^{\frac{4}{3}}\,dx-(\widetilde{C}+2) s_0^{\frac{1}{3}} \cdot m \leq \int_{\mathbb{R}^3} A\left(\sigma_R\right)\, dx \leq E_0(\widetilde{\rho})+\widetilde{C} \int_{\mathbb{R}^3}\left(\sigma_R\right)^{\frac{4}{3}}\, dx
		$$ which again implies $\int_{\mathbb{R}^3}\left(\sigma_R\right)^{\frac{4}{3}}\, dx \leq E_0(\widetilde{\rho}) \leq \max \left\{(\widetilde{C}+2) s_0^{\frac{1}{3}} \cdot m, E_0(\widetilde{\rho})\right\}$.
		
		We claim $\sigma_R$ is continuous in $B_R(0)$ and the following inequalities hold true (notice $\sigma_R(x)>0$ implies $|x|<R$):
		\begin{equation} \label{ELg constraint chap 2}
			E_0{ }^{\prime}\left(\sigma_R\right) \geq \lambda_R \text { where }|x|<R \text { and } \sigma_R(x)<R 
		\end{equation}
		\begin{equation} \label{ELl constraint chap 2}
			E_0{ }^{\prime}\left(\sigma_R\right) \leq \lambda_R \text { where } \sigma_R(x)>0 
		\end{equation}
		Recall we have $E_0^{\prime}\left(\sigma\right)(x)=A^{\prime}\left(\sigma(x)\right)-V_{\sigma}(x)$ by Lemma \ref{diff. of energy}, then we can also write
		\begin{equation} \label{ELAl constraint chap 2}
			A^{\prime}\left(\sigma_R(x)\right) \leq V_{\sigma_R}(x)+\lambda_R \text { where } \sigma_R(x)>0
		\end{equation}
		
		In fact, we first apply (generalized) Lagrange multiplier theorem \cite[Appendix A]{Che26G1} to obtain (\ref{ELg constraint chap 2})(\ref{ELl constraint chap 2})(\ref{ELAl constraint chap 2}) holds true a.e. in the region described above, then we get
		\begin{equation} \label{ELminmax}
			A^{'}\left( {\sigma_{R}(x)} \right) = {\min \left\{ {A^{'}(R),{\max\left\{ {V_{\sigma_{R}}(x) + \lambda_{R},0} \right\}}} \right\}} \text{ a.e. in } B_R(0)
		\end{equation}
		Although the sign of $\lambda_R$ is currently unknown (it will be shown negative later), we can still show $\sigma_R{ }^{\frac{1}{3}} \leq c V_{\sigma_R}$ a.e. on the set where $\sigma_R \geq K_R$ for some $K_R>0$ by absorbing $\lambda_R$ when $\sigma_R$ is large enough \cite[Lemma 3]{AB71}. Similar to the arguments in \cite[Section 2]{Che26G1}, we can show $\sigma_R$ is continuous in $B_R(0)$ and thereby replace almost everywhere (a.e.) by everywhere. 
		
		One may consider using arguments similar to those above (bootstrap method) to show Lemma \ref{uniform bound of norm in single star case}. But we need to realize that this time $\sigma_R^{\frac{1}{3}} \leq c V_{\sigma_R}$ a.e. on the set where $\sigma_R \geq K_R$, here $K_R$ depends on $\lambda_R$ thus may not be a uniform constant.
		
		However, this problem can be solved. We claim $K_R$ is indeed a uniform constant for sufficiently large $R$. We can show $\lambda_R<0$ for sufficiently large $R$. To show $\lambda_R<0$, we choose a suitable region $U_R$ where (\ref{ELg constraint chap 2}) holds that $\lambda_R \leq E_0{ }^{\prime}\left(\sigma_R\right)$, analyze the behaviors of terms in $E_0{ }^{\prime}\left(\sigma_R\right)$ as $R \rightarrow \infty$, and obtain $\lambda_R \leq o\left(R^{-1}\right)-c m R^{-1} <0$ on $U_R$. For example, if $x \in B_R(0)$, since $\sigma_R=0$ outside $B_R(0)$, we have
		\begin{equation}\label{can be relaxed to half mass}
			V_{\sigma_R}(x)=\int_{\mathbb{R}^3} \frac{\sigma_R(y)}{|x-y|} \,dy=\int_{B_R(0)} \frac{\sigma_R(y)}{|x-y|} \,dy \geq \frac{1}{2 R} \int_{B_R(0)} \sigma_R(y) \,dy=\frac{m}{2 R}
		\end{equation}
		which gives us the last term $c m R^{-1}$. One can also see how to get the term $o(R^{-1})$ in \cite[Lemma 5]{AB71}. After showing $\lambda_R<0$, thanks to (\ref{ELAl constraint chap 2}) we have $A^{\prime}\left(\sigma_R\right) \leq V_{\sigma_R}$ for $R$ large enough. Recall Lemma \ref{properties of A} (iii) tells us $\sigma_R ^{\frac{1}{3}} \leq A^{\prime}(\sigma_R)$ when $\sigma_R > N$ for some $N>0$. It implies we can take $K_R = N$ such that it is a uniform constant for $R \geq R_1$, where $R_1$ is sufficiently large. Therefore, we can truly apply bootstrap method to prove $\exists C_1>0$, such that $\left\|\sigma_R\right\|_{L^{\infty}\left(\mathbb{R}^3\right)} \leq C_1$ for $R \geq R_1$. Take $C= \max \{C_1, R_1 \}$, we know $\left\|\sigma_R\right\|_{L^{\infty}\left(\mathbb{R}^3\right)} \leq C$ for $R>0$. This finishes the proof of Lemma \ref{uniform bound of norm in single star case} (Step \ref{1}).
	}
\end{proof}

\begin{remark}\label{F3 importance}
	\ref{F3} is important to show the uniform bound of $\sigma_R$ in $L^{\frac{4}{3}}\left(\mathbb{R}^3\right)$. It helps to show $\int_{\mathbb{R}^3}\left(\sigma_R\right)^{\frac{4}{3}}\,dx \leq s_0^{\frac{1}{3}} \cdot m+\frac{1}{(\widetilde{C}+2)} \int_{\mathbb{R}^3} A\left(\sigma_R\right)\,dx$. If we replace $\frac{4}{3}$ by any smaller number in \ref{F3}, we could not obtain such uniform bound since we would not have $(\widetilde{C}+1) \int_{\mathbb{R}^3}\left(\sigma_R\right)^{\frac{4}{3}}\,dx \leq E_0(\widetilde{\rho})+\widetilde{C} \int_{\mathbb{R}^3}\left(\sigma_R\right)^{\frac{4}{3}}\,dx$ as above, as Proposition \ref{unbounded below} suggests.
	
\end{remark}

To show spt $\sigma_R$ is contained in $\overline{B_{R_1}(0)}$ (Step \ref{3}), thanks to (\ref{ELAl constraint chap 2}), it's sufficient to show $V_{\sigma_R}(x) \leq-\lambda_R$ outside the ball. To do so, we strengthen the result $\lambda_R<0$ and obtain $\lambda_R \leq l<0$ for all large $R$. We first estimate the energy $E_0\left(\sigma_R\right)$, which turns out to satisfies $E_0\left(\sigma_R\right) \leq a<0$ as well.

\begin{lemma}[Uniform Bound for Energies and Lagrange Multipliers {\cite{AB71}}] \label{uniform bound for energy and multipliers}
	$\exists \widetilde{R}_0, \alpha, l$, such that $\forall R>\widetilde{R}_0$, $E_0\left(\sigma_R\right) \leq \alpha<0$ and $\lambda_R \leq l<0$, where $\sigma_R$ is the energy minimizer in $W_R$, and $\lambda_R$ is the corresponding Lagrange multiplier.
\end{lemma}

\begin{proof}
	{
		\rm
		For energies, we choose a text function $\rho_R(x)=R^{-3}\rho (R^{-1}x) \in W_R$ for some suitable function $\rho$, and estimate $E\left(\rho_R\right)=o\left(R^{-1}\right)-c R^{-1}$ as $R \rightarrow \infty$, see \cite[Lemma 6]{AB71}, then $\exists R_2$ such that $\forall R \geq R_2$, $E\left(\sigma_R\right) \leq E\left(\sigma_{R_2}\right) \leq E\left(\rho_{R_2}\right) \coloneq \alpha<0$.
		
		For Lagrange multipliers, by Hölder's inequality, we have
		$$
		-\alpha \leq-E\left(\sigma_R\right) \leq \frac{1}{2} \int_{\mathbb{R}^3} \sigma_R V_{\sigma_R}\,dx \leq \frac{1}{2} m\left\|V_{\sigma_R}\right\|_{L^{\infty}\left(\mathbb{R}^3\right)}
		$$
		then for $R$ large enough, $\left\|V_{\sigma_R}\right\|_{L^{\infty}\left(\mathbb{R}^3\right)} \geq-\frac{2 \alpha}{m}$, which can help to imply $\sigma_R$ cannot become uniformly diffuse as $R \rightarrow \infty$.	That is, we claim $\exists \epsilon_0$, such that $\forall R>0, \exists x_R \in \mathbb{R}^3$, we have
		$$
		\int_{\left|y-x_R\right|<1} \sigma_R(y) \,dy \geq \epsilon_0
		$$

		In fact, let $\epsilon_R=\sup\limits_x \int_{|y-x|<1} \rho_R\,dy$, and estimate $V_{\sigma_R}(x)$ as
		$$
		\begin{aligned}
			V_{\sigma_R}(x) & =\int_{|y-x|<1} \frac{\sigma_R(y)}{|y-x|} \,dy+\int_{1<|y-x|<r} \frac{\sigma_R(y)}{|y-x|} \,dy+\int_{|y-x|>r} \frac{\sigma_R(y)}{|y-x|} \,dy \\
			& :=V_1+V_2+V_3
		\end{aligned}
		$$
		
		It's easy to see $V_3 \leq \frac{M}{r}$. Since the shell $1<|y-x|<r$ can be covered by finite balls of radius 1 (it is 3-dimensional space thus the number of balls can be $\leq c r^3$), then $V_2 \leq c \epsilon_R r^3$. Thanks to Lemma \ref{uniform bound of norm in single star case} $\left(\left\|\sigma_R\right\|_{L^{\infty}\left(\mathbb{R}^3\right)} \leq k_1\right)$ and Proposition \ref{bound of potential}, we have $V_1 \leq c\left(\epsilon_R^a+\epsilon_R^b\right)$ where $a>0, b>0$. Let $g(r)=\frac{M}{r}+c \epsilon_R r^3+c\left(\epsilon_R^a+\epsilon_R^b\right)$, where $r>0$. We know it has minimal value $g_{\min }=\widetilde{c} \epsilon_R^{\frac{1}{4}}+c\left(\epsilon_R^a+\epsilon_R^b\right)$ at $r_{\min }=\left(\frac{M}{3 c \epsilon_R}\right)^{\frac{1}{4}}$. If $r_{\min } \geq 1$, then $\left\|V_{\sigma_R}\right\|_{L^{\infty}\left(\mathbb{R}^3\right)} \leq \widetilde{c} \epsilon_R^{\frac{1}{4}}+c\left(\epsilon_R^a+\epsilon_R^b\right)$. On the other hand, $\left\|V_{\sigma_R}\right\|_{L^{\infty}\left(\mathbb{R}^3\right)} \geq-\frac{2 e}{m}$. Therefore, $\exists \widetilde{\epsilon}>0, \widetilde{R}>0$, such that $\forall R>\widetilde{R}$, we have $\epsilon_R>\widetilde{\epsilon}$. If $r_{\text {min }}<1$, then $\epsilon_R>\frac{M}{3 c}$, with an abuse of notation, $\epsilon_R>\widetilde{\epsilon}$ still holds by letting $\widetilde{\epsilon}=\min \left\{\widetilde{\epsilon}, \frac{M}{3 c}\right\}$. If $R \leq \widetilde{R}, \sigma_R \in W_{\widetilde{R}}$, thus spt $\sigma_R$ is covered by finite balls with radius 1 (the number of balls $C$ is independent of $R$). Since the mass of $\sigma_R$ is $m$, we know $\epsilon_R>\frac{m}{2 C}$. Let $\epsilon=\min \left\{\widetilde{\epsilon}, \frac{m}{2 C}\right\}$, we know $\epsilon_R>\epsilon$ for all $R$, which proves the claim.
		
		By the construction of $W_R$ we get $\sigma_R$ is axially symmetric, then $x_R$ mentioned in the claim should satisfy $r\left(x_R\right) \leq r_0$.
		
		Therefore, we can modify the arguments to prove $\lambda_R<0$ in Step \ref{1} by choosing suitable $x_0 \in B_r\left(x_R\right)$ to get $\lambda_R \leq l<0$ for all large $R$.
	}
\end{proof}

\begin{remark}\label{negative energy and Lagrange multiplier for any minimizer}
	One can use an argument similar to the proof of Lemma \ref{uniform bound for energy and multipliers} above to show that, in fact, for any energy minimizer, the corresponding conclusion holds. That is, the energy and the Lagrangian multiplier are both negative.
\end{remark}

\begin{remark}
	The result that the energy is negative is nice, since later we will show $\sigma_R$ is truly a minimizer on $W_{\infty}=mR{(\mathbb{R}^3)}$, then $E_0\left(\sigma_R\right)$ is the minimal energy $e_0(m)$, which is less than 0.
\end{remark}

\begin{lemma}[Uniform Bound on the Sizes of Supports {\cite{AB71}}]\label{uniform bound on the sizes in single star case}
	$\exists R_1>0$, such that $\forall R>R_1$, spt $\sigma_R$ is contained in $\overline{B_{R_1}(0)}$.
\end{lemma}

\begin{proof}
	{
		\rm
		We first show $V_{\sigma_R}(x) \leq-l$ holds when the distance from $x$ to the z-axis $r(x)$ and $R$ is large enough, then due to (\ref{ELAl constraint chap 2}) and Lemma \ref{uniform bound for energy and multipliers} we know $\sigma_R(x)=0$. Let $a,b>0$ be as in the estimate $V_1 \leq c(\epsilon_R^a + \epsilon_R^b)$ from the proof of Lemma \ref{uniform bound for energy and multipliers}, and let $c_1,c_2,c_3>0$ be sufficiently large constants. Then define
		\[
		f(r, \pi)=c_1(\pi^{-a}+\pi^{-b})+c_2\pi^{-1}r^3+c_3r^{-1}
		\]
		Notice if we choose $r\geq r_0:=1+\frac{-3 c_3}{l}$, then $c_3 r^{-1}<-\frac{l}{3}$, and we have $f(r, \pi) \leq-l$ when $\pi$ is large enough.
		
		Now we claim $\exists \pi_0>2 r_0-1$, $\forall \pi>\pi_0$, $\forall x$ with $r(x)>\pi$, $\forall r \in\left(r_0, \frac{1+\pi}{2}\right)$ we have $V_{\sigma_R}(x) \leq f(r, \pi) \leq-l$.
		
		In fact, let $\epsilon_{\pi, R}=\sup\limits_{r(x)>\frac{\pi}{2}} \int_{|y-x|<1} \sigma_R(y) \,dy$, then $\exists x_{\pi, R} \in \mathbb{R}^3$ with $r\left(x_{\pi, R}\right)>\frac{\pi}{2}$, such that
		$$
		\int_{\left|y-x_{\pi, R}\right|<1} \sigma_R(y) \,dy>\epsilon_{\pi, R}-\frac{1}{\pi}
		$$
		On the other hand, due to the axial symmetry of $\sigma_R(y), \sigma_R(y)$ has mass $\geq {cr}\left(x_{\pi, R}\right)\left(\epsilon_{\pi, R}-\frac{1}{\pi}\right)$ in the torus obtained by revolving the ball $\left|y-x_{\pi, R}\right|<1$ around the z-axis, which implies $\epsilon_{\pi, R} \leq \frac{m}{{cr}\left(x_{\pi, R}\right)}+\frac{1}{\pi}<\frac{2 m}{c \pi}+\frac{1}{\pi}$. By choosing $c$ sufficiently large, we have $\epsilon_{\pi, R}<\frac{c}{\pi}$.
		
		For all $x$ with $r(x)>\pi$, notice due to $r<\frac{1+\pi}{2}$, we can choose finite balls centered at $x_n$ of radius 1 with $r\left(x_n\right)>\frac{\pi}{2}$ covering $1<|y-x|<r$, and the number of balls can be $\leq c r^3$. Similar to Lemma \ref{uniform bound for energy and multipliers}, together with $\epsilon_{\pi, R}<\frac{c}{\pi}$, we have
		$$
		\begin{aligned}
			V_{\sigma_R}(x) & =\int_{|y-x|<1} \frac{\sigma_R(y)}{|y-x|} \,dy+\int_{1<|y-x|<r} \frac{\sigma_R(y)}{|y-x|} \,dy+\int_{|y-x|>r} \frac{\sigma_R(y)}{|y-x|} \,dy \\
			& \leq c\left(\epsilon_{\pi, R}^a+\epsilon_{\pi, R}^b\right)+c r^3 \epsilon_{\pi, R}+c r^{-1} \\
			& \leq c_1\left(\pi^{-a}+\pi^{-b}\right)+c_2 \pi^{-1} r^3+c_3 r^{-1} \\
			& =f(r, \pi)
		\end{aligned}
		$$
		Since $r \geq r_0$, when $\pi$ is large enough we have $f(r, \pi) \leq-l$ as we mention above. Therefore, the claim holds true. Thus $\forall \pi> \pi_0$, $\forall x$ with $r(x)>\pi$, we have $V_{\sigma_R}(x) \leq f(\frac{2r_0+1+\pi}{4}, \pi) \leq-l$, then $\sigma_R(x)=0$ due to (\ref{ELAl constraint chap 2}) and Lemma \ref{uniform bound for energy and multipliers}.
		
		Now we consider the z-direction and show $V_{\sigma_R}(x) \leq-l$ holds when $z(x)$ is large enough, which is relatively easier. We can first show $\exists \delta>0$, $r \geq 2$ and $\widetilde{R_0}>0$, such that $\forall R \geq \widetilde{R_0}$, if $\int_{|z(x)-d|<r} \sigma_R(x) \,dx<\delta$, then $\sigma_R(x)=0$ for $|z(x)-d| \leq 1$. This comes from that when $|z(x)-d| \leq 1$ we have the estimate:
		$$
		V_{\sigma_R}(x)=\int_{|z(y)-d|<r} \frac{\sigma_R(x)}{|y-x|} \,dy+\int_{|z(y)-d|>r} \frac{\sigma_R(x)}{|y-x|} \,dy \leq c\left(\delta^a+\delta^b\right)+\frac{m}{r-1}
		$$
		
		And then letting $\delta$ small enough and $r$ large enough can give $V_{\sigma_R}(x) \leq-l$, which means $\sigma_R(x)=0$. Then we set $Z_n=\{x|| z(x)-2 n \mid \leq 1\},  n=0, \pm 1, \pm 2, \ldots$. and $Z_n^{\prime}= \{x|| z(x)-2 n \mid<r\}$, without loss of generality, we can take $r$ to be an integer. Let $N$ be the number of $Z_n^{\prime}$ such that $\int_{Z_n^{\prime}} \sigma_R\,dx \geq \delta$. Since the mass of $\sigma_R$ is $m<\infty$ and the $Z_n^{\prime}$s overlap finite times $\widetilde{r}<\infty$, we have $N \delta \leq \widetilde{r} m$, i.e., $N \leq \frac{\widetilde{r} m}{\delta}$. Therefore, $\sigma_R(x)=0$ for large $R$ except on a set of horizontal slabs $Z_n$ of total height at most $\frac{2 \widetilde{r} m}{\delta}$.

		If these $Z_n{s}$ were known to be contiguous, then we would have $\sigma_R(x)=0$ for $|z(x)| \geq \frac{2 r m}{\delta}$. But if the $Z_n{s}$ were not contiguous, we could follow the idea of the strong rearrangement inequality in Lieb \cite[Lemma 3]{Lie77}. We could rearrange the order of $Z_n$, slide the half-space above an empty $Z_n$ down one unit and obtain a new $\widetilde{\sigma_R} \in W_R$ from $\sigma_R$. During the process we notice $$U\left(\widetilde{\sigma_R}\right)=\int_{\mathbb{R}^3} A\left(\widetilde{\sigma_R}\right)\,dx=\int_{\mathbb{R}^3} A\left(\sigma_R\right)\,dx=U\left(\sigma_R\right)$$
		$$G\left(\widetilde{\sigma_R}, \widetilde{\sigma_R}\right)=\int_{\mathbb{R}^3} \widetilde{\sigma_R} V_{\widetilde{\sigma_R}}\,dx>\int_{\mathbb{R}^3} \sigma_R V_{\sigma_R}\,dx=G\left(\sigma_R, \sigma_R\right)$$
		then we obtain $E_0\left(\widetilde{\sigma_R}\right)<E_0\left(\sigma_R\right)$, which leads to a contradiction.
		
		Collect the results above, we know we can choose a large $R_1$, such that $\forall R>R_1$, spt $\sigma_R$ is contained in $\overline{B_{R_1}(0)}$.
	}
\end{proof}

Now we can go forward to Step \ref{4} and prove the constrained minimizer can be a global minimizer.

\begin{lemma}[Existence of Global Minimizer in Non-rotating case {\cite{AB71}}]\label{existence of global minimizer in single star case}
	Take $R^{\prime}= \max \left\{k_1, R_1\right\}+1$, where $k_1$ and $R_1$ are given in Lemma \ref{uniform bound of norm in single star case} and Lemma \ref{uniform bound on the sizes in single star case}, then we have $\forall R \geq R^{\prime}, \sigma_R$ is a minimizer on $m {R}\left(\mathbb{R}^3\right)$.
\end{lemma}

\begin{proof}
	{
		\rm
		Thanks to Lemma \ref{uniform bound of norm in single star case} and Lemma \ref{uniform bound on the sizes in single star case}, we know $\forall R \geq R^{\prime}, \sigma_R \in W_{R^{\prime}}$, which implies $E_0\left(\sigma_R\right)=E_0\left(\sigma_{R^{\prime}}\right)$ since they should both be minimizers on $W_{R^{\prime}}$. Thus, we just need to show $\sigma_{R^{\prime}}$ is a minimizer on $W_{\infty}=m {R}\left(\mathbb{R}^3\right)$. To show that, given $\rho \in W_{\infty}$, we construct a $\rho_R \in W_R$ for large $R$, with $E_0\left(\rho_R\right) \leq E_0(\rho)+\epsilon$, where $\epsilon$ can be arbitrarily small. (For the construction of $\rho_R$, although $\rho$ is not necessarily to be spherically symmetric, but one can apply a strong rearrangement inequality in Lieb \cite[Lemma 3]{Lie77} and the generalization by Sobolev of the Hardy-Littlewood theorem on rearrangements of functions \cite{Sob38} to replace $\rho$ by another spherically symmetric and radially decreasing (after translation) $\rho$, see for example \cite[Appendix]{AB71}, \cite[Section 6]{McC06}. Then one can refer to \cite[Section 6]{AB71} or \cite[Section 4]{Che26E3} to construct $\rho_R$.) Then we have $E_0\left(\sigma_{R^{\prime}}\right)=E_0\left(\sigma_R\right) \leq E_0\left(\rho_R\right) \leq E_0(\rho)+\epsilon$. Let $\epsilon \rightarrow 0$ we obtain $E_0\left(\sigma_{R^{\prime}}\right) \leq E_0(\rho)$, which shows the result.
	}
\end{proof}

It is natural to conjecture whether there are some relations between Lagrange-multipliers $\lambda_m$ and the derivatives of $e_0(m)$, since they can be obtained by somehow ``differentiating'' $E_0\left(\sigma_m\right)$. Thanks to Lieb and Yau's paper \cite[Theorem 3 and Theorem 4]{LY87}, we can know the left and right derivatives of $e_0(m)$ bound $\lambda: e_0^{\prime}\left(m^{+}\right) \leq \lambda \leq e_0^{\prime}\left(m^{-}\right)$, as cited in \cite[Theorem 3.5]{McC06} and \cite[Section 2]{Che26G1}. Moreover, we find a more precise quantitative property in the case that complies with polytropic equations of state in the following.

\begin{proposition}[Expression for Lagrange-multipliers] \label{exact value of multiplier}
	Let pressure $P(\rho)$ satisfies polytropic law with index $\gamma > \frac{4}{3}$, $E_0(\rho)$ from (\ref{energy}), $e_0(m)$ from (\ref{e_0}) and $m \in[0,1]$, then $e_0^{\prime}(m)=\lambda_m=-(5 \gamma-6) m^{\frac{2 \gamma-2}{3 \gamma-4}} U(\sigma)<0$, where $\sigma$ is the minimizer of $E_0(\rho)$ with mass $1$.
\end{proposition}

\begin{proof}
	{
		\rm
		By scaling relations in Section \ref{section3-scaling relations} (Theorem \ref{scaling relation for energy}) we can obtain $e_0(m)=m^{\frac{5 \gamma-6}{3 \gamma-4}} e_0$, which implies $e_0^{\prime}$ exists. Then $e_0^{\prime}(m)=\lambda_m$ can come from Theorem \ref{non-rotating} (viii). As in the proof of Theorem \ref{non-rotating} (ii), we use the scaling relations mentioned in Section \ref{section3-scaling relations} (Proposition \ref{scaling relation for derivatives}). We obtain $E_0^{\prime}\left(\sigma_m\right)(x)=\frac{1}{A^{\gamma-1}} E_0^{\prime}(\sigma)\left(\frac{1}{B} x\right)$, thus $\lambda_m= \frac{1}{A^{\gamma-1}} \lambda_1$. Since $e_0(m)=m^{\frac{5 \gamma-6}{3 \gamma-4}} e_0$, then $e_0^{\prime}(m)=\frac{5 \gamma-6}{3 \gamma-4} m^{\frac{2 \gamma-2}{3 \gamma-4}} e_0$, and $e_0^{\prime}(m)=\frac{1}{A^{\gamma-1}} e_0^{\prime}(1)$. Therefore, we just need to consider the case $m=1$. Notice we have Euler-Lagrange equation $E_0^{\prime}(\sigma)(x)=\lambda_1$ a.e. on $\{\sigma(x)>0\}$. Rather than substituting a specific $x$ to calculate $E_0^{\prime}(\sigma)$, we view $E_0^{\prime}(\sigma)$ as a functional in the sense that $E_0^{\prime}(\sigma)(\rho)=\int_{\mathbb{R}^3} E_0^{\prime}(\sigma) \rho\,dx$ as in Lemma \ref{diff. of energy} or Remark \ref{diff. at mini.}. We take $\rho=\sigma$ as a test function, notice the mass of $\sigma$ is 1. Then we have
		$$
		\int_{\mathbb{R}^3} E_0^{\prime}(\sigma) \sigma\,dx=\int_{\{\sigma(x)>0\}} E_0^{\prime}(\sigma) \sigma\,dx=\lambda_1 \int_{\{\sigma(x)>0\}} \sigma\,dx=\lambda_1
		$$
		
		On the other hand, we have
		$$
		\int_{\mathbb{R}^3} E_0^{\prime}(\sigma) \sigma\,dx=\int_{\mathbb{R}^3}\left(A^{\prime}(\sigma)-V_\sigma\right) \sigma\,dx=\gamma U(\sigma)-G(\sigma, \sigma)
		$$
		
		Then we get $\lambda_1=\gamma U(\sigma)-G(\sigma, \sigma)$, while $e_0^{\prime}(1)=\frac{5 \gamma-6}{3 \gamma-4} e_0=\frac{5 \gamma-6}{3 \gamma-4}\left(U(\sigma)-\frac{G(\sigma, \sigma)}{2}\right)$. Due to the previous arguments that $\lambda_1= e_0^{\prime}(1)$, we can deduce that $\lambda_1= e_0^{\prime}(1)=(6-5 \gamma) U(\sigma)$ if $\gamma \neq 2$. In the case $\gamma = 2$, we can apply Remark \ref{relation between G and U} below and obtain $\frac{G(\sigma, \sigma)}{2}=(3 \gamma-3) U(\sigma)$, then we also have $\lambda_1= e_0^{\prime}(1)=(6-5 \gamma) U(\sigma)$.
	}
\end{proof}

\subsection{Uniqueness Results for Non-rotating Bodies}\label{subsection2.3-uniqueness results for non-rotating bodies}

In this subsection, we discuss uniqueness results. To show the uniqueness result in the following, we follow Lieb and Yau's arguments \cite[Lemma 10, Lemma 11 and remark following]{LY87} \cite{McC06}, but adapt it from the framework of quantum mechanics to that of classical mechanics. We show if, in addition, $P(\rho)$ satisfies \ref{F4}, $A^{\prime}\left(\rho^3\right)$ has second derivative and is convex, uniqueness of minimizer up to translation can also be shown in the following arguments. In particular, if polytropic law holds true with $\gamma> \frac{4}{3}$, then $A^{\prime}\left(\rho^3\right)=\frac{K \gamma}{\gamma-1} \rho^{3(\gamma-1)}$ satisfies those conditions.

Due to Theorem \ref{non-rotating} (iii), we know the minimizer $\sigma_m$ is spherically symmetric after translation. In general, given a spherically symmetric function $\rho$, with an abuse of notation, we denote $\rho(|x|)=\rho(s)$, and 
$$m(r)=m_\rho(r)=\int_{|x|<r} \rho\,dx=4 \pi \int_0^r \rho(s) s^2 d s$$
is the mass of $\rho$ in $\{|x|<r\}$, then easy to show (we assume the following computational results are finite. This condition holds, for instance, when $\rho$ is continuous with compact support):
\begin{equation} \label{V_r}
	V_\rho(x)=V_\rho(|x|)=\frac{m(|x|)}{|x|}+4 \pi \int_{|x|}^{\infty} t \rho(t) d t
\end{equation}

By direct computation we have
\begin{equation}\label{diff. of V_r outside 0}
	\frac{d}{d r} V_\rho(r)=\frac{-m(r)}{r^2}, \text { where } r \neq 0 
\end{equation}
\begin{equation}\label{gradient of V outside 0}
	\nabla V_\rho(x)=\frac{d}{d r} V_\rho(|x|) \cdot \frac{x}{|x|}=\frac{-m(|x|)}{|x|^2}\cdot \frac{x}{|x|}, \text { where } x \neq 0
\end{equation}

Note that $V_\rho(0)=V_\rho(|0|)=4\pi\int_0^\infty t\rho(t)dt$, which follows directly from the definition of $V_\rho$ (or by applying L'Hôpital's rule to show $V_\rho(0)=\lim\limits_{x \rightarrow 0} V_\rho(x)=4 \pi \int_{0}^{\infty} t \rho(t) d t$). Then after using L'Hôpital's rule \cite[Theorem 5.13]{Rud76}, one can also show that
\begin{equation}\label{diff. of V_r at 0}
	\frac{d}{d r} V_\rho(0)=0 
\end{equation}
\begin{equation}\label{gradient of V at 0}
	\nabla V_\rho(0)=0
\end{equation}

Another way to show $\nabla V_\rho(0)=0$ is the fact that since $\lim\limits_{x \rightarrow 0} \nabla V_\rho(x)=0$ exists, by mean value theorem we can prove $\nabla V_\rho(0)=\lim\limits_{x \rightarrow 0} \nabla V_\rho(x)=0$. The third way to show it is using the fact $V_\rho$ is also spherically symmetric (same for $\frac{d}{d r} V_\rho(0)$)\footnote{Strictly speaking, we need to verify that $V$ is differentiable at $0$ then use the symmetric property. This differentiability can be deduced from the regularity of $\rho$. For instance, if $\rho$ is continuous with compact support, then arguments similar to those in Theorem \ref{non-rotating} (vi) ensure the differentiability of $V$ at $0$.}.

Based on (\ref{V_r}), we can show the Shell Theorem, which is given by Newton.
\begin{theorem} [Shell Theorem {\cite{She24} \cite[Page 192 Theorem XXX and Page 193 Theorem XXXI]{New87} \cite[Page 218 Theorem XXX and Theorem XXXI]{NCW99}}]\label{Shell Theorem}
	Suppose $\rho$ is spherically symmetric, if $\rho$ is supported in $B_R(0)$, denote its total mass by $m = \int_0^R 4\pi r^2 \rho(r)\,dr$, then the gravitational potential $V_\rho$ satisfies:
	\begin{enumerate}[(i)]
		\item For exterior points ($|x| \ge R$):
		$V_\rho(x) = \frac{m}{|x|}$.
		That is, given an object $P$ outside $B_R(0)$, $\rho$ produces the same external gravitational potential as if all $\rho$'s mass were concentrated at its center.
		\item For interior points ($|x| \le R$):
		$V_\rho(x) = \frac{m(|x|)}{|x|}+4 \pi \int_{|x|}^{\infty} t \rho(t) d t$,
		where the second term shows that the outer shells contribute a constant potential inside. In particular, if $\rho$ vanishes in $B_r(0)$ for some $r<R$, then $V_\rho$ is a constant in $B_r(0)$. Given an object $P$ inside $B_r(0)$, no net gravitational force is exerted by $\rho$ on $P$ due to (\ref{diff. of V_r outside 0}).
	\end{enumerate}
\end{theorem}

\begin{remark}
	In fact, one can also utilize Legendre polynomials to prove the Shell Theorem Theorem \ref{Shell Theorem}. The idea is essentially the same as the proof of Corollary \ref{gravitational potential energy between 2 spherical bodies} in the following.
\end{remark}

\begin{corollary}[Gravitational Potential Energy between Two Spherical Bodies]\label{gravitational potential energy between 2 spherical bodies}
	Suppose $\rho$ and $\sigma$ are spherically symmetric with total masses $m_1$ and $m_2$, and their centers are separated by a distance $R$.
	If $R$ is greater than the sum of their radii (so that the spheres do not overlap), then their mutual gravitational potential energy $G(\rho, \sigma)$ is
	$$G(\rho, \sigma) = \frac{m_1 m_2}{R}$$
	In other words, two non-overlapping spherically symmetric bodies attract each other exactly as if all their mass were concentrated at their centers.
\end{corollary}

\begin{proof}
	Without loss of generality, we can set up a coordinate system such that the center of $\rho$ is $(0,0,0)$, and the center of $\sigma$ is $\mathbf D = (0,0,R)$. Assume the radius of $\rho$ is $r_1$ and the radius of $\sigma$ is $r_2$, notice $R\geq r_1+r_2$, then thanks to Theorem \ref{Shell Theorem} the mutual gravitational potential energy is given by:
	
	\begin{equation}\label{mutual gravitational potential energy}
		G(\rho, \sigma) = \int_{B_{r_1}(0)} \rho(x)\,V_\sigma(x)\,dx
		= m_2 \int_{B_{r_1}(0)} \frac{\rho(x)}{|x-\mathbf D|}\,dx
	\end{equation}

	We can expand $1/|x-\mathbf D|$ as:
	
	\begin{equation}\label{expansion of 1 over distance}
		\frac{1}{|x-\mathbf D|}
		= \frac{1}{R} \sum_{\ell=0}^\infty \left(\frac{|x|}{R}\right)^\ell P_\ell(\cos\gamma)
	\end{equation}
	
	where $P_\ell$ is the Legendre's polynomial, $\gamma$ is the angle between $x$ and $ \mathbf D$.

	Then plug (\ref{expansion of 1 over distance}) into (\ref{mutual gravitational potential energy}), and notice $\rho$ is spherically symmetric, which implies the Legendre polynomials expansion of $\rho$ is itself — $\rho(x)=\rho(|x|)$), then we know all terms involving $P_\ell(\cos \gamma)$ for $\ell>0$ vanish upon angular integration, leaving only the $\ell=0$ term:
	
	$$
	G(\rho, \sigma) = m_2 \int_{B_{r_1}(0)} \frac{\rho(x)}{|x-\mathbf D|}\,dx=\frac{m_2}{R}\int_{B_{r_1}(0)} \rho(x)\,dx= \frac{m_1 m_2}{R}
	$$
	
	This is the result we want.
\end{proof}

\begin{remark} \label{homogeneous case}
	Notice that in Corollary \ref{gravitational potential energy between 2 spherical bodies}, it is not necessary to assume constant density, i.e., the celestial body need not be homogeneous. In the case of a homogeneous celestial body, we may alternatively use the fact that $\frac{1}{|x-\mathbf D|}$ is a harmonic function in $B_{r_1}(0)$ and apply the mean value property of harmonic function to prove Corollary \ref{gravitational potential energy between 2 spherical bodies}.
\end{remark}

Let's recall that for minimizer $\sigma$, by Theorem \ref{non-rotating} (vii) we have the Euler-Lagrange equation (\ref{EL}):

\begin{equation}
	A^{\prime}(\sigma(x))=\left[V_{\sigma}(x)+\lambda\right]_{+} \tag{EL}
\end{equation}

In particular,

\begin{equation} \label{EL no subscript}
	E_0^{\prime}(\sigma)(x)=\lambda \text {, where } \sigma(x)>0
\end{equation}

More specifically,
\begin{equation} \label{EL r}
	E_0^{\prime}(\sigma)(r)=A^{\prime}(\sigma(r))-V_\rho(r)=\lambda \text {, where } \sigma(r)>0
\end{equation}

Due to (\ref{V_r}), after taking the derivative of (\ref{EL r}), we can have
\begin{equation}\label{diff. of A'(r)}
	\frac{d}{d r} A^{\prime}(\sigma(r))=-r^{-2} m(r) \text {, where } \sigma>0
\end{equation}

Hence $\sigma$ is decreasing when it is positive. Moreover, from (\ref{EL no subscript}) we also have $\sigma(x)=\left(A^{\prime}\right)^{-1}\left(V_\sigma(x)+\lambda\right)$, notice that $\Delta\left(V_\sigma+\lambda\right)=\Delta V_\sigma=-4 \pi \sigma$, set $\Theta_\sigma=V_\sigma+\lambda$, then we have
\begin{equation} \label{laplacian equation for Theta when positive sigma}
	-\Delta \Theta_\sigma=4 \pi \sigma=4 \pi\left(A^{\prime}\right)^{-1}\left(\Theta_\sigma\right) \text {, where } \sigma>0
\end{equation}

We relax the constraints and allow the equation to hold over the entire $\mathbb{R}^3$.
\begin{equation}\label{laplacian equation for Theta in whole domain}
	-\Delta \Theta_\sigma=4 \pi\left(A^{\prime}\right)^{-1}\left(\Theta_\sigma\right)
\end{equation}

\begin{remark} \label{relation between sigma and Theta}
	As mentioned in \cite{LY87}, since $\sigma$ is spherically symmetric, so is $\Theta_\sigma$. Thus (\ref{laplacian equation for Theta when positive sigma}) and (\ref{laplacian equation for Theta in whole domain}) turn out to be a second order ODE. Due to (\ref{gradient of V at 0}) we have $\nabla \Theta_\sigma(0)=0$. If $\Theta_\sigma(0)=\beta>0$, by uniqueness theorem of ODE (for example, Picard's method), we know (\ref{laplacian equation for Theta in whole domain}) has a unique positive solution up to some $R(\beta)$ and spt $\Theta_\sigma=\overline{B_{R(\beta)}(0)}$. Since $\Delta \Theta_\sigma=-4 \pi\left(A^{\prime}\right)^{-1}\left(\Theta_\sigma\right)= -4 \pi \sigma$, we can see $B_{R(\beta)}(0)=\{\sigma>0\}$. Therefore, $0<R(\beta)<\infty$ due to Theorem \ref{non-rotating} (v). 
	
	On the other hand, thanks to Lieb and Yau \cite[Lemma 8]{LY87}, we can show $0<R(\beta)<\infty$ holds true for any radial solutions to (\ref{laplacian equation for Theta in whole domain}) under certain assumptions of initial conditions mentioned above and of pressure \cite{LY87}. Furthermore, note that our preceding arguments are not restricted to a specific mass (e.g., $m=1$), but naturally generalizes to any mass $m>0$. Consequently, the positive solutions to (\ref{laplacian equation for Theta in whole domain}) actually correspond one-to-one with the solutions to (\ref{EL}). We will return to discuss more about these relations later.
\end{remark}

Although the solutions to (\ref{EL}) can correspond one-to-one with the positive solutions to (\ref{laplacian equation for Theta in whole domain}), and one positive solution to (\ref{laplacian equation for Theta in whole domain}) is unique given initial conditions by uniqueness theorem of ODE, so far we cannot conclude that the minimizer is unique. The problem is we do not yet know two minimizers with same mass have to correspond with the same initial conditions of (\ref{laplacian equation for Theta in whole domain}). The next several results help to show the minimizer is truly unique.

\begin{lemma}[Relation between Central Density and Mass {\cite[Lemma 10]{LY87}}] \label{relation between central density and mass}
	Suppose $\sigma_m$ and $\sigma_{\widetilde{m}}$ are minimizers for $E_0(\rho)$ with $\int_{\mathbb{R}^3} \sigma_m\,dx=m$, $\int_{\mathbb{R}^3} \sigma_{\widetilde{m}}\,dx=\widetilde{m}$ respectively and their centers of mass are $(0,0,0)^T$. Let $R_m$ and $R_{\widetilde{m}}$ be the radii of their supports and set $R=\max \left\{R_m, R_{\widetilde{m}}\right\}$. Suppose that $\sigma_m(0)>\sigma_{\widetilde{m}}(0)$, then $\forall 0<r<R$ we have $m(r)>\widetilde{m}(r)$.
\end{lemma}

\begin{proof}
	{
		\rm
		We provide a proof outline here. One can check the proofs in \cite{LY87} and references therein, to understand some step(s) omitted in the following.
		
		We use proof by contradiction. Since $\sigma_m$ and $\sigma_{\widetilde{m}}$ are minimizers, by Theorem \ref{non-rotating} (vi) we know they are continuous. It is easy to see $m>0$ and $R_m>0$, so is $R$. Thus $\sigma_m(0)>\sigma_{\widetilde{m}}(0)$ implies $\exists \hat{r}$, $\forall 0<r<\hat{r}$ we have $m(r)> \widetilde{m}(r)$. If the result of the Lemma does not hold, there is an $r_0<R$ such that $m\left(r_0\right) \leq \widetilde{m}\left(r_0\right)$. By continuity of $m(r)$ and $\widetilde{m}(r)$, there exists an $\widetilde{r}$, where $0<\hat{r} \leq \widetilde{r} \leq r_0$, such that $Q:=m(\widetilde{r})=\widetilde{m}(\widetilde{r})$. We first claim that $\sigma_m(\widetilde{r}) \neq \sigma_{\widetilde{m}}(\widetilde{r})$. If not, that is $\sigma_m(\widetilde{r})=\sigma_{\widetilde{m}}(\widetilde{r})>0$, due to (\ref{gradient of V outside 0}) (\ref{diff. of A'(r)}) (\ref{laplacian equation for Theta when positive sigma}), we have $\Theta_{\sigma_m}(\widetilde{r})=\Theta_{\sigma_{\widetilde{m}}}(\widetilde{r})>0$, $\dot{\Theta}_{\sigma_m}(\widetilde{r})=\dot{\Theta}_{\sigma_{\widetilde{m}}}(\widetilde{r})$. By the uniqueness theorem of ODE, $\Theta_{\sigma_m}=\Theta_{\sigma_{\widetilde{m}}}$ where they are positive. Apply (\ref{laplacian equation for Theta when positive sigma}) we obtain $\sigma_m(0)=\sigma_{\widetilde{m}}(0)$ which contradicts that $\sigma_m(0)> \sigma_{\widetilde{m}}(0)$.
		
		We consider dividing the density functions into inner parts and outer parts.
		That is, let 
		\[
		\sigma_m^i = \sigma_m \mathbf{1}_{\{|x| \leq \bar{r}\}},\quad 
		\sigma_m^o = \sigma_m \mathbf{1}_{\{|x| > \bar{r}\}},
		\]
		\[
		\sigma_{\widetilde{m}}^i = \sigma_{\widetilde{m}} \mathbf{1}_{\{|x| \leq \bar{r}\}},\quad 
		\sigma_{\widetilde{m}}^o = \sigma_{\widetilde{m}} \mathbf{1}_{\{|x| > \bar{r}\}}.
		\]
		Set $E_0^i\left(\rho^i\right)=E_0\left(\rho^i\right)$, and
		$$
		E_0^o\left(\rho^o\right)=U\left(\rho^o\right)-\frac{G\left(\rho^o, \rho^o\right)}{2}-Q \int_{\mathbb{R}^3} \frac{\rho^o(x)}{|x|} \,dx
		$$
		By Remark \ref{Shell Theorem}, $\int_{\mathbb{R}^3} \sigma_m^i\,dx=\int_{\{|x| \leq \widetilde{r}\}} \sigma_m^i\,dx=Q$, and notice
		\begin{align*}
			G\left(\sigma_m^i, \sigma_m^o\right) 
			&= \iint_{\mathbb{R}^3 \times \mathbb{R}^3} \frac{\sigma_m^i(y) \sigma_m^o(x)}{|x-y|} \,dxdy \\
			&= \int_{\{|y| \leq \widetilde{r}\}} \sigma_m^i(y) V_{\sigma_m^o}(y) \,dy \\
			&\stackrel{(\ref{V_r})}{=} \int_{\{|y| \leq f\}} \sigma_m^i(y) V_{\sigma_m^o}(0) \,dy \\
			&= Q V_{\sigma_m^o}(0) \\
			&= Q \int_{\mathbb{R}^3} \frac{\sigma_m^o(x)}{|x|} \,dx.
		\end{align*}
		we obtain $e_0(m)=E_0\left(\sigma_m\right)=E_0^i\left(\sigma_m^i\right)+E_0^o\left(\sigma_m^o\right)$. Similarly we have $e_0(\widetilde{m})=E_0\left(\sigma_{\widetilde{m}}\right)= E_0^i\left(\sigma_m^i\right)+E_0^o\left(\sigma_{\widetilde{m}}^o\right)$. 
		
		Moreover, we claim $E_{0}^{i}\left( \sigma_{m}^{i} \right) = E_{0}^{i}\left( \sigma_{\widetilde{m}}^{i} \right) = {\inf\limits_{\rho^{i} \in U^{i}}{E_{0}^{i}\left( \rho^{i} \right)}}$, where 
		$$U^{i} = \left\{ \rho^{i} \in L^{\frac{4}{3}}\left( \mathbb{R}^{3} \right) \middle| \rho^{i} \geq 0,\rho^{i} \text{ vanishes outside } \overline{B_{\widetilde{r}}(0)},~{\int_{\mathbb{R}^3}\rho^{i}\,dx} = Q \right\}$$ 
		
		In fact, if not, we can find a $\rho^i\in U^i$ with $E_{0}\left( \rho^{i} \right) < E_{0}^{i}\left( \sigma_{m}^{i} \right)$, then $\tilde{\rho}=\rho^i+\sigma_m^o$ has mass $m$ and energy $E_{0}\left( \widetilde{\rho} \right) = E_{0}\left( \rho^{i} \right) + E_{0}^{o}\left( \sigma_{m}^{o} \right) < E_{0}\left( \sigma_{m} \right)$, which contradicts the fact $\sigma_m$ is a minimizer with mass $m$. Now we consider $\widetilde{\sigma}=\sigma_m^i+\sigma_{\widetilde{m}}^o$, we also have $e_0(\widetilde{m})=E_0(\widetilde{\sigma})$ which means $\widetilde{\sigma}$. is also a minimizer with mass $\widetilde{m}$. But due to $\sigma_m(\widetilde{r}) \neq \sigma_{\widetilde{m}}(\widetilde{r})$, we conclude $\widetilde{\sigma}$ is not continuous at $\widetilde{r}$, which violates the regularity of the minimizer proved in Theorem \ref{non-rotating}. Another way to reach a contradiction is to note that $\hat{\sigma}=\sigma_m^i+\sigma_m^o$ is a minimizer with mass $m$. One of $\widetilde{\sigma}$ and $\hat{\sigma}$ must be increasing (jumping up) at $\widetilde{r}$, and this contradicts the symmetric decreasing property of minimizers in Theorem \ref{non-rotating}.
	}
\end{proof}

\begin{remark}
	Notice that in Lemma \ref{relation between central density and mass}, only when $\forall 0<r<R$ we have $m(r)>\widetilde{m}(r)$. Hence it does not say $m>\widetilde{m}$, although we shall later see that this is true (Corollary \ref{relation between central density and mass strenthed} or Remark \ref{asymptotic behaviour of radius and norm in single star case}) after obtaining the uniqueness result. If we knew in advance that $m>\widetilde{m}$, the proof of the uniqueness of minimizer would be trivial.
\end{remark}


So far we only know the energy consists of inertial energy and gravitational interaction energy:
\begin{equation}\label{E_0}
	E_0(\rho)=U(\rho)-\frac{G(\rho, \rho)}{2}
\end{equation}
To find a deeper relation between energy and $A(\rho)$, it will be nice if we can find another expression of $E_0(\rho)$ in which $G(\rho, \rho)$ does not appear. Thanks to Lieb and Yau's arguments \cite{LY87}, we can truly find such expression when $\rho$ is a minimizer, since in this case we have Euler-Lagrange equation which connects $A^{\prime}(\rho)$ with gravitational potential $V_\rho$ and thereby $G(\rho, \rho)$.

\begin{lemma}[Relation between $E_0$ and $A$ {\cite[Lemma 6]{LY87}}]\label{relation between E_0 and A}
	Let $P(s)$ satisfies \ref{F1} \ref{F2} \ref{F3}. Suppose $\sigma$ is continuous, spherically symmetric with finite mass $m=\int_{\mathbb{R}^3} \sigma\,dx$, $\sigma$ has compact support, and satisfies (\ref{EL}) on all of $\mathbb{R}^3$ for $J=0$ and a single $\lambda<0$. Then
	\begin{equation}\label{E_0 by A}
		E_0(\sigma)=\int_{\mathbb{R}^3}\left(4 A(\sigma)-3 \sigma A^{\prime}(\sigma)\right)\,dx
	\end{equation}
\end{lemma}

\begin{proof}
	{
		\rm
		If $\sigma=0$, then (\ref{E_0 by A}) is trivial since both sides are 0. If $\sigma$ is not zero function, since $\sigma \in L^1\left(\mathbb{R}^3\right) \cap C_c^0\left(\mathbb{R}^3\right)$, by Proposition \ref{diff. of poten.} we know $V_\rho$ is continuously differentiable, thus we can take the derivative of (\ref{EL r}) where $\sigma>0$ (notice $\{\sigma>0\}$ is open). Since $\sigma$ is spherically symmetric, we have exactly (\ref{diff. of A'(r)}): 
		$$\frac{d}{d r} A^{\prime}(\sigma(r))=-r^{-2} m(r)$$ 
		Multiply it by $4 \pi r^3 \sigma(r)$ and integrate over $\mathbb{R}$, we obtain 
		$$-4 \pi \int_0^{\infty} r^3\left(\frac{d}{d r} A^{\prime}(\sigma)\right) \sigma d r=4 \pi \int_0^{\infty} m(r) \sigma(r) r d r$$
		Easy to see $\iint_{\{|x| \leq|y|\}} \frac{\sigma(x) \sigma(y)}{|x-y|} \,dx\,dy=\iint_{\{|x| \geq|y|\}} \frac{\sigma(x) \sigma(y)}{|x-y|} \,dx\,dy$. Then we compute
		
		$$
		\begin{aligned}
			\frac{G(\sigma, \sigma)}{2} &=\frac{1}{2}\left(\iint_{\{|x| \leq|y|\}} \frac{\sigma(x) \sigma(y)}{|x-y|} \,dx\,dy+\iint_{\{|x| \geq|y|\}} \frac{\sigma(x) \sigma(y)}{|x-y|} \,dx\,dy\right) \\
			& =\iint_{\{|x| \leq|y|\}} \frac{\sigma(x) \sigma(y)}{|x-y|} \,dx\,dy \\
			& =\iint_{\mathbb{R}^3\times \mathbb{R}^3} \frac{\sigma(x) \sigma(y)\mathbf{1}_{\{|x| \leq|y|\}}}{|x-y|}\,dx\,dy \\
			& =\int_{\mathbb{R}^3} \sigma(y) \,dy \int_{\mathbb{R}^3} \frac{\sigma(x) \mathbf{1}_{\{|x| \leq|y|\}}}{|x-y|} \,dx \\
			& =\int_{\mathbb{R}^3} \sigma(y) V_{\left\{\sigma \mathbf{1}_{\{|\cdot| \leq|y|\}}\right\}}(y) \,dy \\
			& \stackrel{(\ref{V_r})}{=} \int_{\mathbb{R}^3} \frac{\sigma(|y|) m(|y|)}{|y|}  \,dy \\
			& =4 \pi \int_0^{\infty} m(r) \sigma(r) r d r
		\end{aligned}
		$$
		
		Notice we use spherical coordinate transformation in the last identity.
		
		On the other hand, we notice 
		$$\left(\frac{d}{d r} A^{\prime}(\sigma)\right) \sigma=\frac{d}{d r}\left(\sigma A^{\prime}(\sigma)-A(\sigma)\right)$$
		
		 And $\sigma$ vanishes at 0 and outside a large ball, so is $\sigma A^{\prime}(\sigma)-A(\sigma)$. Then we use spherical coordinate transformation and integration by parts, and obtain
		
		$$
		\begin{aligned}
			3 \int_{\mathbb{R}^3}\left(\sigma A^{\prime}(\sigma)-A(\sigma)\right)(x) \,dx & =12 \pi \int_0^{\infty}\left(\sigma A^{\prime}(\sigma)-A(\sigma)\right) r^2 d r \\
			& =-4 \pi \int_0^{\infty} r^3\left(\frac{d}{d r} A^{\prime}(\sigma)\right) \sigma d r
		\end{aligned}
		$$
		
		Collect the results above we obtain $E_0(\sigma)=\int_{\mathbb{R}^3} A(\sigma)\,dx-\frac{G(\sigma, \sigma)}{2}=\int_{\mathbb{R}^3}(4 A\left(\sigma)-3 \sigma A^{\prime}(\sigma)\right)\,dx$.
	}
\end{proof}

\begin{remark}\label{relation between E_0 and A for minimizer}
	Although the differentiability of $\sigma$ is used in Lieb and Yau's proof \cite{LY87}, we refine their proof as above and do not necessarily need the differentiability of $\sigma$. This is the reason we do not need to assume \ref{F4} holds. If $P(s)$ satisfies \ref{F1}\ref{F2}\ref{F3}, then thanks to Theorem \ref{non-rotating}, a minimizer $\sigma_m$ (after translation) meets the conditions in Lemma \ref{relation between E_0 and A}, thus (\ref{E_0 by A}) holds true for $\sigma_m$.
\end{remark}

\begin{remark}\label{relation between G and U}
	In particular, if $P(\sigma)$ satisfies the polytropic equations of state, then $\sigma A^{\prime}(\sigma)=\gamma A(\sigma)=\frac{K \gamma}{\gamma-1} \sigma^\gamma$. Thanks to Lemma \ref{relation between E_0 and A}, we have $E_0(\sigma)=(4-3 \gamma) U(\sigma)=(4-3 \gamma) \int_{\mathbb{R}^3} A(\sigma)\,dx$, and $G(\sigma, \sigma)=(6 \gamma-6) U(\sigma)$. Actually, inspired by Lieb and Yau \cite[Remark after Lemma 6]{LY87}, we claim 
	$$G(\sigma, \sigma)=(6 \gamma-6) U(\sigma)$$ 
	for minimizers using the scaling relations similar to those in Section \ref{section3-scaling relations}. That is, (with an abuse of notations) let $\sigma_\lambda(x)=\lambda^3 \sigma(\lambda x)$ we have $\int_{\mathbb{R}^3} \sigma_\lambda\,dx=\int_{\mathbb{R}^3} \sigma\,dx$ thus $E_0(\sigma) \leq E_0\left(\sigma_\lambda\right)$, $\left.\frac{d}{d \lambda} E_0\left(\sigma_\lambda\right)\right|_{\lambda=1}=0$. Moreover, $E_0\left(\sigma_\lambda\right)=\lambda^{3 \gamma-3} U(\sigma)-\frac{\lambda}{2} G(\sigma, \sigma)$, then we have $\left.\frac{d}{d \lambda} E_0\left(\sigma_\lambda\right)\right|_{\lambda=1}=(3 \gamma-3) U(\sigma)-\frac{G(\rho, \rho)}{2}=0$, and hence the claim follows.
\end{remark}

In order to obtain uniqueness result for minimizers, Lieb and Yau \cite[Lemma 11 and Remark following]{LY87} use a proof by contradiction. They said the only property of $A(s)$ used was the convexity of $g$ (also see \cite[Section 3]{McC06}), where $g(s)=4 A(s)-3 A^{\prime}(s) s$. However, a closer look reveals that the proof needs the strict inequality $g''(s) < 0$ on some interval, whereas convexity only gives $g''(s) \leq 0$. Nevertheless, it turns out we indeed can show that the convexity of $g$, together with the specific structure of $A$, implies $g^{\prime \prime}<0$ holds true in a small interval, which is already enough to make a contradiction in Proposition \ref{uniqueness in non-rotating cases}.

In the following lemma, instead of considering $g$, we consider $f(s):=A^{\prime}\left(s^3\right)$ and one can check $g^{\prime \prime}(s)<0$ is equivalent to $f^{\prime \prime}(s)>0$ when $s>0$.
\begin{lemma}[Positive Second Order Derivative of $A^{\prime}\left(s^3\right)$]\label{positive 2nd order derivative}
	If $P(s)$ satisfies \ref{F1}\ref{F2}\ref{F3}\ref{F4}, $A(s) \in C^3\left(\mathbb{R}^{+}\right)$(that is $A(s)$ has third order continuous derivative when $s>0)$, and $f(s):=A^{\prime}\left(s^3\right)$ is convex. Then $\forall \delta>0, \exists s_\delta \in(0, \delta)$, such that $f^{\prime \prime}\left(s_\delta\right)>0$.
\end{lemma}
\begin{proof}
	{
		\rm
		Since $P^{\prime}(s)=s A^{\prime \prime}(s)$, we know $A(s) \in C^3\left(\mathbb{R}^{+}\right)$ is equivalent to $P(s) \in C^2\left(\mathbb{R}^{+}\right)$. Notice $f^{\prime}(s)=3 s^2 A^{\prime \prime}\left(s^3\right)=\frac{3 P^{\prime}\left(s^3\right)}{s}$, and $f^{\prime \prime}(s)=\frac{9 s^2 P^{\prime \prime}\left(s^3\right) \cdot s-3 P^{\prime}\left(s^3\right)}{s^2}$. Since $f(s)$ is convex, $f^{\prime \prime}(s)=\frac{9 s^2 P^{\prime \prime}\left(s^3\right) \cdot s-3 P^{\prime}\left(s^3\right)}{s^2} \geq 0$. If the conclusion does not hold, then $\exists \delta_0>0$, such that $\forall s \in\left(0, \delta_0\right), f^{\prime \prime}(s)=0$, which implies $3 t P^{\prime \prime}(t)-P^{\prime}(t)=0$ for all $t \in \left(0, \delta_0^3\right)$, then $P^{\prime}(t)=C t^{\frac{1}{3}}$ for some $C$ when $t \in\left(0, \delta_0^3\right)$. Then $\lim\limits_{t \rightarrow 0} \frac{P^{\prime}(t)}{t^{\frac{1}{3}}}=C$. Since we already know $P(t)$ satisfies \ref{F2}, that is $\lim\limits_{t \rightarrow 0} \frac{P(t)}{t^{\frac{4}{3}}}=0$, then $P \rightarrow 0$ as $t \rightarrow 0$. Then we can apply L'Hôpital's rule \cite[Theorem 5.13]{Rud76} and obtain $$0=\lim\limits_{t \rightarrow 0} \frac{P(t)}{t^{\frac{4}{3}}}= \lim\limits_{t \rightarrow 0} \frac{3}{4} \cdot \frac{P^{\prime}(t)}{t^{\frac{1}{3}}}=\frac{3 C}{4}$$
		which means $C=0$. Then $P^{\prime}(t)=C t^{\frac{1}{3}}=0$ when $t \in\left(0, \delta_0^3\right)$, notice again by (\ref{F2}) $\lim\limits_{t \rightarrow 0} \frac{P(t)}{t^{\frac{4}{3}}}=0$, hence $P(t)=0$. It contradicts the assumption \ref{F1} that $P(t)$ is strictly increasing for $t>0$. Therefore, we know $\forall \delta>0, \exists s_\delta \in(0, \delta)$, such that $f^{\prime \prime}\left(s_\delta\right)>0$.	
	}
\end{proof}

\begin{remark} \label{positive 2nd order derivative in a interval}
	Since $f^{\prime \prime}$ is continuous, by Lemma \ref{positive 2nd order derivative} we know $\forall \delta>0$, $\exists\left(s_1, s_2\right) \subset(0,2 \delta)$ such that $f^{\prime \prime}(s)>0$ if $s \in\left(s_1, s_2\right)$.
\end{remark}

\begin{proposition}[Uniqueness of Minimizer in Non-rotating Cases \cite{LY87} \cite{McC06}] \label{uniqueness in non-rotating cases}
	If $P(s)$ satisfies \ref{F1}\ref{F2}\ref{F3}\ref{F4}, $A(s) \in C^3\left(\mathbb{R}^{+}\right)$, and $A^{\prime}\left(s^3\right)$ is convex, then the minimizer of $E_0(\rho)$ is unique up to translation. In particular, if $P(s)$ satisfies the polytropic equations of state with index $\gamma>\frac{4}{3}$, the minimizer of $E_0(\rho)$ is unique up to translation.
\end{proposition}

\begin{proof}
	{
		\rm
		We provide a proof outline here. One can check the proofs in \cite{LY87} and references therein, to understand some step(s) omitted in the following.
		
		Let $f(s)=A^{\prime}\left(s^3\right)$, and $g(s)=4 A(s)-3 A^{\prime}(s) s$, by direct computation we obtain $f^{\prime \prime}(s)=6 s A^{\prime \prime}\left(s^3\right)+9 A^{\prime \prime \prime}\left(s^3\right) s^4$, and $g^{\prime \prime}(s)=-2 A^{\prime \prime}(s)- 3 s A^{\prime \prime \prime}(s)$ when $s>0$. (Notice $A(s) \in C^3\left(\mathbb{R}^{+}\right)$ guarantees those derivatives exist and are continuous.) When $s>0$, $f$ is convex then $f^{\prime \prime} \geq 0$ implies $g^{\prime \prime} \leq 0$, then $g$ is concave. Therefore, we have $\frac{g(a)-g(b)}{a-b}\left\{\begin{array}{l}\leq g^{\prime}(b), b<a \\ \geq g^{\prime}(b), b>a\end{array}\right.$, therefore, $\forall a, b \geq 0$, we have $g(a)-g(b) \leq g^{\prime}(b)(a-b)$.
		
		If $m=0$, then minimizer $\sigma_m=0$ is unique. If $m>0$, without loss of generality, we assume $m=$ 1. Given two minimizers $\sigma$ and $\widetilde{\sigma}$, thanks to Theorem \ref{non-rotating}, after translation, we can assume $\sigma$ and $\widetilde{\sigma}$ are continuous, spherically symmetric, centered at $0$, differentiable when they are positive, and radially decreasing. Then we know $\sigma(0)>0, \widetilde{\sigma}(0)>0$, and $\nabla \sigma(0)=\nabla \widetilde{\sigma}(0)=0$ (due to being symmetric and differentiable at $0$). If $\sigma(0)=\widetilde{\sigma}(0)$, due to (\ref{gradient of V at 0}) (\ref{laplacian equation for Theta when positive sigma}) we have $\Theta_\sigma(0)=\Theta_{\widetilde{\sigma}}(0)$, and $\dot{\Theta}_\sigma(0)=\dot{\Theta}_{\widetilde{\sigma}}(0)$. By the uniqueness theorem of ODE, $\Theta_\sigma=\Theta_{\widetilde{\sigma}}$ where they are positive. Due to (\ref{laplacian equation for Theta when positive sigma}) we obtain $\sigma=\widetilde{\sigma}$. If $\sigma(0) \neq \widetilde{\sigma}(0)$, without loss of generality, we assume $\sigma(0)>\widetilde{\sigma}(0)>0$. Let $R>0$ and $\widetilde{R}>0$ be the radii of their supports. By Lemma \ref{relation between central density and mass}, $m(r)>\widetilde{m}(r)$ for all $0<r<\max \{R, \widetilde{R}\}$. Then $R \leq \widetilde{R}$; otherwise $m=\widetilde{m}(\widetilde{R})<m(\widetilde{R})<m$, which is a contradiction. Then since $\sigma$ and $\widetilde{\sigma}$ are minimizers, thanks to Lemma \ref{relation between E_0 and A}, use spherical coordinate transformation then we have
		$$
		\begin{aligned}
			0&=E_0(\sigma)-E_0(\widetilde{\sigma})\\
			&=\int_{\mathbb{R}^3} g(\sigma)-g(\widetilde{\sigma})\,dx\\
			&=4 \pi \int_0^{\widetilde{R}}(g(\sigma(r))-g(\widetilde{\sigma}(r))) r^2 d r \\
			&\leq 4 \pi \int_0^{\widetilde{R}} g^{\prime}(\widetilde{\sigma}(r))(\sigma(r)-\widetilde{\sigma}(r)) r^2 d r
		\end{aligned}
		$$

		Notice that since $P(\rho)$ satisfies \ref{F4}, $\sigma$ and $\widetilde{\sigma}$ are continuously differentiable when they are positive due to reasons similar to those mentioned in the proof of Theorem \ref{non-rotating} (vi). Integrating the last integral by parts and using the definitions of $m(r)$ and $\widetilde{m}(r)$, we have $0 \leq -\int_0^{\widetilde{R}}(m(r)-\widetilde{m}(r)) g^{\prime \prime}(\widetilde{\sigma}(r)) \dot{\widetilde{\sigma}}(r) d r$. Notice on $(0, \widetilde{R})$, we have, $g^{\prime \prime}(\widetilde{\sigma}(r)) \leq 0$, 
		$m(r)>\widetilde{m}(r)$. By Theorem \ref{non-rotating} (vii) we know $\widetilde{\sigma}(r)=\phi \circ W_{\widetilde{\sigma}}$ where $\widetilde{\sigma}$ is positive, here $\phi=\left(A^{\prime}\right)^{-1}$ and $W_{\widetilde{\sigma}}=V_{\widetilde{\sigma}}+\lambda$. Since $P$ satisfies \ref{F4}, by Lemma \ref{properties of A} we know $\phi$ is differentiable with $\phi^{\prime}(s)=\frac{1}{A^{\prime \prime}(\phi(s))}>0$. From (\ref{diff. of V_r outside 0}) we have $\frac{d}{d r} W_{\widetilde{\sigma}}= \frac{-\widetilde{m}(r)}{r^2}<0$. Therefore, $\dot{\widetilde{\sigma}}(r)<0$ holds for all $r \in(0, \widetilde{R})$. By Lemma \ref{positive 2nd order derivative} and Remark \ref{positive 2nd order derivative in a interval}, together with the relation between $f^{\prime \prime}$ and $g^{\prime \prime}$ we know $\exists\left(a_1, a_2\right) \subset (0, \widetilde{\sigma}(0))$, such that $g^{\prime \prime}(s)<0$ when $s \in\left(a_1, a_2\right)$. 
		
		On the other hand, notice $\widetilde{\sigma}(\widetilde{R})=0$, by continuous of $\widetilde{\sigma}$ we know $\exists\left(r_1, r_2\right) \subset(0, \widetilde{R})$, such that $g^{\prime \prime}(\widetilde{\sigma}(r))<0$. Therefore, we get $0 \leq -\int_0^{\widetilde{R}}(m(r)-\widetilde{m}(r)) g^{\prime \prime}(\widetilde{\sigma}(r)) \dot {\widetilde{\sigma}}(r) d r<0$, which leads to a contradiction.
	}
\end{proof}

\begin{corollary}\label{relation between central density and mass strenthed}
	Let $\sigma_m$ be the minimizer for mass $m$, then the mass $m$ can be viewed as a strictly increasing function of minimizers' central density.
\end{corollary}

\begin{proof}
	{
		\rm
		If $\sigma_m$ and $\sigma_{m^{\prime}}$ are minimizers for mass $m$ and $m^{\prime}$ respectively, thanks to Theorem \ref{non-rotating} (iii) we can assume they are spherically symmetric and radially decreasing, and $\sigma_m(0)$ and $\sigma_{m^{\prime}}(0)$ are their central densities. If $\sigma_m(0)>\sigma_{m^{\prime}}(0)$, due to Lemma \ref{relation between central density and mass} we know $m \geq m^{\prime}$. But Proposition \ref{uniqueness in non-rotating cases} tells us that $m=m^{\prime}$ is impossible unless $\sigma_m=\sigma_{m^{\prime}}$. Thus, we have $m>m^{\prime}$.
	}
\end{proof}

Now we know the minimizer satisfies (\ref{EL}) and it is unique. A natural question is: given a radial solution $\sigma_{(\alpha)}$ to (\ref{EL}) with central density (initial data) $\sigma_{(\alpha)}(0)=\alpha$, is it a minimizer? It turns out the answer is yes, but we first need to discuss more about Remark \ref{relation between sigma and Theta} to tie together minimizers, radial solutions to (\ref{EL}) and positive radial solutions to (\ref{laplacian equation for Theta in whole domain}). We first review a lemma in \cite{LY87} (but for simplicity, here we assume the pressure satisfies the polytropic equations of state).

\begin{lemma}[Compact Support of $\Theta$ {\cite[Lemma 8]{LY87}}]\label{compact of positive Theta}
	Assume the pressure satisfies the polytropic equations of state with index $\gamma>\frac{3}{2}$. Let $\Theta$ be the radial solution(i.e., $\Theta$ is spherically symmetric) to (\ref{laplacian equation for Theta in whole domain}) obtained by integrating outwards from $r=$ 0 with the initial conditions $\Theta(0)=\beta>0$ and $\dot{\Theta}(0)=0$. Then there is some $R(\beta) \in (0, \infty)$ such that $\Theta(R(\beta))=0$. For $r<R(\beta), \dot{\Theta}(r)<0$. In particular, $\Theta (r)$ is decreasing up to $R(\beta)$.
\end{lemma}

\begin{proof}
	{
		\rm
		$\Theta(0)=\beta>0$ implies $R(\beta) \neq 0$ if $R(\beta)$ exists. Set $\rho=\left(A^{\prime}\right)^{-1}(\Theta)$, by Shell Theorem, or the arguments similar as (\ref{V_r}) and (\ref{gradient of V outside 0}) we have $\Theta^{\prime}(r)=-4 \pi \int_0^r \frac{t^2}{r^2} \rho(t) d t$, from which we get $\Theta^{\prime}(r)<0$ for $r<R(\beta)$. Integrate it again and use Fubini's theorem, we have
		$$
		\Theta(r)=\beta-4 \pi \int_0^r\left(\frac{1}{t}-\frac{1}{r}\right) t^2 \rho(t) d t
		$$
		
		Suppose $R(\beta)=\infty$, set $g(t)=t \rho(t)$, we could obtain the result that $g \in L^1([0, \infty), d r)$ \cite[Lemma 8]{LY87}. However, one can also show $\Theta(r) \geq \frac{c}{r}$ for large $r$ \cite[Lemma 8]{LY87}. Then for $r$ large enough we have $$\rho= \left(A^{\prime}\right)^{-1}(\Theta)=C \Theta^{\frac{1}{\gamma-1}} \geq c r^{-\frac{1}{\gamma-1}}$$ 
		Notice $\gamma>\frac{3}{2}$ implies $1-\frac{1}{\gamma-1}>-1$, thus $\int_0^{\infty} g(r) d r=\int_0^{\infty} r \rho(r) d r=\infty$, which make a contradiction. For more details one can refer to \cite[Lemma 8 and Proposition 9]{LY87}.
	}
\end{proof}


\begin{remark}\label{nonpositive of Theta after R(beta)}
	Actually $\Theta^{\prime}(r)=-4 \pi \int_0^r \frac{t^2}{r^2} \rho(t) d t$ implies $\Theta^{\prime}(r)<0$ when $\Theta$ is positive. Therefore, by contradiction arguments we can show $\Theta(r) \leq 0$ when $r \geq R(\beta)$, where $R(\beta)$ is given in Lemma \ref{compact of positive Theta}.
\end{remark}

Due to Remark \ref{relation between sigma and Theta}, we know given a radial solution $\sigma_{(\alpha)}$ to (\ref{EL}) with compact support, we can find a positive radial solution $\Theta_{\sigma_{(\alpha)}}$ to (\ref{laplacian equation for Theta in whole domain}). On the other hand, Lemma \ref{compact of positive Theta} and Remark \ref{nonpositive of Theta after R(beta)} tell us a radial solution $\Theta$ to (\ref{laplacian equation for Theta in whole domain}) with positive initial data is only positive up to $r=R(\beta)$. Let 

$$\sigma_\Theta(r)=\left\{\begin{array}{r}\frac{-\Delta \Theta}{4\pi}\text{,    }r \leq R(\beta) \\ 0\text{,    }r >R(\beta) \end{array}\right.$$

Then one can see $\sigma_\Theta$ satisfies (\ref{EL}) with $\lambda = \Theta(0)-V_{\sigma_\Theta}(0)$.

We claim any radial solution $\sigma_{(\alpha)}$ to (\ref{EL}) has compact support. In fact, due to (\ref{diff. of A'(r)}) we know $\sigma_{(\alpha)}$ is decreasing when it is positive. If its support is unbounded, then $\sigma_{(\alpha)}(r)>0$ for all $r \geq 0$. Then we can construct a solution $\Theta_{\sigma_{(\alpha)}}$ to (\ref{laplacian equation for Theta in whole domain}) with $\sigma_{(\alpha)}=\left(A^{\prime}\right)^{-1}\left(\Theta_{\sigma_{(\alpha)}}\right)$ (one can also see (\ref{laplacian equation for Theta when positive sigma})), which implies $\Theta_{\sigma_{(\alpha)}}$ also has unbounded support. But it contracts Remark \ref{nonpositive of Theta after R(beta)}.

Therefore, we can actually remove the compact support assumption at the beginning in this remark that ``a radial solution $\sigma_{(\alpha)}$ to (\ref{EL}) with compact support", and show the solutions to (\ref{EL}) indeed correspond one-to-one with the positive solutions to (\ref{laplacian equation for Theta in whole domain}).

Moreover, from $\lambda = \Theta(0)-V_{\sigma_\Theta}(0)$ above, one can see the choice of $\lambda$ in (\ref{EL}) is not arbitrary but is related to the choice of $\Theta(0)$, which in turn is related to the central density $\sigma_{(\alpha)}(0)=\alpha$ (Recall $\sigma_{(\alpha)}=\left(A^{\prime}\right)^{-1}\left(\Theta_{\sigma_{(\alpha)}}\right)$).

Based on the above observation, we can now establish the equivalent relation between radial solutions to (\ref{EL}) (with specific central densities) and minimizers:

\begin{proposition}[Equivalence between (\ref{EL}) Solutions and Minimizers {\cite[Lemma 12]{LY87}}]\label{equivalence between solutions and minimizers}
	Let $\sigma_{(\alpha)}$ be the unique nonnegative radial solution of (\ref{EL}) with central density $\sigma(0)=\alpha$, then $\sigma_{(\alpha)}$ is the unique minimizer for $E_0(\rho)$ among $mR(\mathbb{R}^3)$, where $m=\int_{\mathbb{R}^3}\sigma_{(\alpha)}\,dx$. That is, all radial solutions of (\ref{EL}) parametrized by their central density are in fact minimizers of $E_0(\rho)$ among $mR(\mathbb{R}^3)$ for some mass $m$.
\end{proposition}

\begin{proof}	
	
	Here we provide a proof outline. One can refer to \cite[Lemma 12]{LY87} to understand some step(s) omitted here.
	We first note that the one-to-one correspondence between a central density of zero and a total mass of zero is trivial, as both correspond to the vacuum (zero function). Hence we can now turn to the case where mass is positive.
	
	The idea of Lieb and Yau is: let $G=(0,\infty)$, and
	$$D=\{\alpha \mid \text{The solution } \sigma_{(\alpha)} \text{ to } \eqref{EL} \text{ is a minimizer with some total mass } \|\sigma_{(\alpha)}\|_{L^1} \in G\}$$ 
	For each $m\in G$, we know by Theorem \ref{non-rotating} and Proposition \ref{uniqueness in non-rotating cases} there is a unique minimizer $\rho_m$, and a unique central density $\alpha_m$ (see also Corollary \ref{relation between central density and mass strenthed}). We let $\Gamma: G \rightarrow D$ denote this map from $m\in G$ to $\alpha_m \in D$.
	
	Then in {\cite[Lemma 12]{LY87}},  one can show $\Gamma$ (or $\Gamma^{-1}$) is a homeomorphism between $G$ and $D$, and $D$ turns out to be closed in $\mathbb{R}^+$. $G$ is connected, so is $D$, hence $D$ is a closed interval in $\mathbb{R}^+$. Since $\Gamma^{-1}$ is a homeomorphism, then $D$ can only be $\mathbb{R}^+$.
\end{proof}

\begin{remark}
	If assumption \ref{F3} is replaced by \ref{F3'}, then $G$ in the proof of Proposition \ref{equivalence between solutions and minimizers} should be $(0,m(K))$, where $m(K)$ is the Chandrasekhar mass mentioned in Remark \ref{collapse}.
\end{remark}
\section{Scaling Relations between Stars with Different Mass} \label{section3-scaling relations}

The goal of this section is to provide a description of scaling method and scaling relations between solutions with different mass, which can help us to find more quantitative properties. We have been using those results in the previous arguments and will continue to use them in the following. We assume the polytropic equation of state $P(\rho)=K\rho^\gamma$ with index $\gamma>\frac{4}{3}$ holds true from this section.

\subsection{Scaling relations between solutions to partial differential equations}\label{subsection3.1-Scaling relations between solutions to partial differential equations}

It is known that the Navier-Stokes equations (and many other PDEs) have a natural scaling. The (incompressible) Navier-Stokes equations have the form:
\begin{equation}\label{NS}
	\left\{\begin{array}{l}
		\partial_t u+(u \cdot \nabla) u+\nabla \pi=v \Delta u \\
		\operatorname {div}(u)=0
	\end{array}\right. \tag{NS}
\end{equation}
where $\pi$ is the pressure and $v>0$ is the viscosity coefficient. For simplicity we assume $v=1$ is constant. Formally, we take the divergence of the first equation and use the second equation to eliminate some terms and obtain $\Delta \pi=-\operatorname{div}((u \cdot \nabla) u)$, i.e. $$\operatorname{div}((u \cdot \nabla) u)+\operatorname {div}(\nabla \pi)=0$$
Then we can introduce Leray projection operator $\mathbb{P}=I d+\nabla(-\Delta)^{-1} \operatorname {div}$ \cite{Cob24} to get rid of the pressure term from the first equation and obtain:
\begin{equation}\label{NS'}
	\left\{\begin{array}{l}
		\partial_t u+\mathbb{P}((u \cdot \nabla) u)=\Delta u \\
		\operatorname {div}(u)=0
	\end{array}\right. \tag{NS'}
\end{equation}

For any function $f(t, x)$, we denote $f_\lambda(t, x):=\lambda f\left(\lambda^\alpha t, \lambda^\beta x\right)$. Let a function $u(t, x)$ be a solution to (\ref{NS'}), notice that $\mathbb{P}\left(f_\lambda\right)=(\mathbb{P}(f))_\lambda$, then by direct computation we have
$$
\partial_t u_\lambda+\mathbb{P}\left(\left(u_\lambda \cdot \nabla\right) u_\lambda\right)-\Delta u_\lambda=\lambda^\alpha\left(\partial_t u\right)_\lambda+\lambda^{1+\beta}(\mathbb{P}((u \cdot \nabla) u))_\lambda-\lambda^{2 \beta}(\Delta u)_\lambda
$$

Therefore, when $\alpha=2$, $\beta=1$, $u_\lambda(t, x)=\lambda u\left(\lambda^2 t, \lambda x\right)$ is also a solution to (\ref{NS'}). Actually $\forall c \in \mathbb{R}$, $c u\left(\lambda^2 t, \lambda x\right)$ is a solution since (\ref{NS'}) is linear with respect to $c$.

However, when we consider the reduced Euler-Poisson equations (\ref{EP'}) with $\omega=0$, or the corresponding Euler-Lagrange equations (\ref{EL}) with $J=0$ (non-rotating case), since the pressure $P(\rho)=K \rho^\gamma$ or the function $A^{\prime}(\rho)=\frac{K}{\gamma-1} \rho^\gamma$ is not linear and there is an extra term $V_\rho, c$ may not be chosen arbitrarily. With an abuse of notation, we assume $g_{A, B}(x)=A g(B x)$. Let a function $\sigma(x)$ be a solution to (\ref{EL}), i.e., $A^{\prime}(\sigma(x))=\left[V_\sigma(x)+\lambda\right]_{+}$, then we have $$A^{\prime}\left(\sigma_{\widetilde{A}, B}(x)\right)=\frac{K \gamma}{\gamma-1}\left(\sigma_{\widetilde{A}, B}(x)\right)^{\gamma-1}=\widetilde{A}^{\gamma-1} A^{\prime}(\sigma(B x))$$
$$V_{\sigma_{A, B}}(x)=\int_{\mathbb{R}^3} \frac{\widetilde{A} \sigma(B y)}{|x-y|} \,dy=\widetilde{A} B^{-2} \int_{\mathbb{R}^3} \frac{\sigma(B y)}{|B x-B y|} \,d(B y)=\widetilde{A} B^{-2} V_\sigma(B x)$$ 
Hence if $\widetilde{A}^{\gamma-1}=\widetilde{A} B^{-2}$ (i.e., $B=\widetilde{A}^{\frac{2-\gamma}{2}}$), we will have $\sigma_{\widetilde{A}, B}$ is also a solution to (\ref{EL}): $$A^{\prime}\left(\sigma_{\widetilde{A}, B}(x)\right)=\left[V_{\sigma_{\bar{A}, B}}(x)+\lambda_{\widetilde{A}, B}\right]_{+}$$
with $\lambda_{\widetilde{A}, B}=\widetilde{A} B^{-2}$.

Moreover, if we know $g$ has mass $m$ (again without loss of generality we assume $m \neq 0$ unless otherwise specified) and hope $g_{A, B}$ has mass 1. Then $\int_{\mathbb{R}^3} g_{A, B}\,dx=A B^{-3}$ $\int_{\mathbb{R}^3} g(B x) \,d(B x)=A B^{-3} m=1$. Therefore, we can solve $A=m^{-\frac{2}{3 \gamma-4}}, B=m^{\frac{\gamma-2}{3 \gamma-4}}$.

\begin{theorem}[Relations between Solutions to Equations with Different Mass] \label{scaling relation for PDEs}
	Let a function $g(x)$ with mass $m$ be a solution to (\ref{EL}) with $J=0$ and Lagrange multiplier $\lambda$, or to (\ref{EP'}) with $\omega=0$, then $g_{A, B}(x)$ is a solution to (\ref{EL}) with $J=0$ and Lagrange multiplier $\lambda_{A, B}$, or to (\ref{EP'}) with $\omega=0$, where $A=m^{-\frac{2}{3 \gamma-4}}, B=m^{\frac{\gamma-2}{3 \gamma-4}}$. Moreover, $g_{A, B}(x)$ has mass 1.
\end{theorem}

\begin{proof}
	{
		\rm
		For the case of Euler-Lagrange equations (\ref{EL}), one can see the arguments above. For the case of the reduced Euler-Poisson equations (\ref{EP'}), notice (\ref{EP'}) becomes 
		$$(\nabla P(g(x)))-g(x)\left(\nabla V_g(x)\right)=0$$ 
		We also notice
		$$\left(\nabla P\left(g_{A, B}(x)\right)\right)-g_{A, B}(x)\left(\nabla V_{g_{A, B}}(x)\right)=A^\gamma B\left(\nabla_{B x} P(g(B x))\right)-A^2 B^{-1} g(B x)\left(\nabla_{B x} V_g(B x)\right)$$
		thus the arguments are essentially the same since we have ${A}^{\gamma-2}B^{2}=1$.
	}
\end{proof}

Therefore, to solve (\ref{EL}), we can also assume a solution has mass 1. Notice the radial solution with certain mass is actually a minimizer and unique (see Proposition \ref{uniqueness in non-rotating cases} and Proposition \ref{equivalence between solutions and minimizers}) and vice versa. It can help us to discuss more properties and relations between minimizers with different mass in the next subsection.

\subsection{Scaling relations between minimizers of $E_0(\rho)$ with different mass}\label{subsection3.2-Scaling relations between minimizers of energy with different mass}

In this subsection we will show that given a minimizer, the corresponding scaling density can also be a minimizer, and then we will discuss more quantum properties.

Recall (\ref{energy}) that $E_0(\rho)=U(\rho)-\frac{G(\rho, \rho)}{2}=\int_{\mathbb{R}^3} A(\rho)\,dx-\frac{1}{2} \iint_{{\mathbb{R}^3}\times {\mathbb{R}^3}} \frac{\rho(x) \rho(y)}{|x-y|} \,dxdy$, where $A(\rho)=\frac{K}{\gamma-1} \rho^\gamma$.

Let $\sigma$ be in the admissible class \eqref{admissible class}, then we have
$$\begin{aligned}U\left(\sigma_{\widetilde{A}, B}\right)&=\int_{\mathbb{R}^3} A\left(\sigma_{\widetilde{A}, B}\right)\,dx\\
	&=\int_{\mathbb{R}^3} \frac{K}{\gamma-1}(\widetilde{A} \sigma(B x))^\gamma\,dx\\
	&=\frac{\widetilde{A}^\gamma}{B^3} \int_{\mathbb{R}^3} \frac{K}{\gamma-1}(\sigma(B x))^\gamma \,d(B x)\\
	&=\widetilde{A}^\gamma B^{-3} U(\sigma)\end{aligned}$$
$$
\begin{aligned}
	G\left(\sigma_{\widetilde{A}, B}, \sigma_{\widetilde{A}, B}\right)&=\iint_{{\mathbb{R}^3}\times {\mathbb{R}^3}} \frac{\sigma_{\widetilde{A}, B}(x) \sigma_{\widetilde{A}, B}(y)}{|x-y|} \,dxdy\\
	&=\widetilde{A}^2 B^{-5} \iint_{\mathbb{R}^3\times \mathbb{R}^3} \frac{\sigma(B x) \sigma(B y)}{|B x-B y|} \,d(B x) \,d(B y) \\
	&=\widetilde{A}^2 B^{-5} G(\sigma, \sigma) 
\end{aligned}
$$
$$\begin{aligned}
	\int_{\mathbb{R}^3} \sigma_{\widetilde{A}, B}\,dx&=\widetilde{A} B^{-3} \int_{\mathbb{R}^3} \sigma(B x) \,d(B x)\\
	&=\widetilde{A} B^{-3} \int_{\mathbb{R}^3} \sigma\,dx\end{aligned}$$

Assume $\sigma_m$ is a minimizer of $E_0(\rho)$ on $m {R}\left(\mathbb{R}^3\right)$, if $A^\gamma B^{-3}=A^2 B^{-5}$ (i.e., $B=A^{\frac{2-\gamma}{2}}$), we will have $\left(\sigma_m\right)_{A, B}$ is a minimizer of $A^\gamma B^{-3} E_0(\rho)$ on $A B^{-3} m {R}\left(\mathbb{R}^3\right)$ (one can show it by contradiction arguments), and thereby a minimizer of $E_0(\rho)$ on $A B^{-3} m {R}\left(\mathbb{R}^3\right)$. If we hope $\left(\sigma_m\right)_{A, B}$ has mass $1$, then $A B^{-3} m=1$. Therefore, we can solve $A=m^{-\frac{2}{3 \gamma-4}}, B=m^{\frac{\gamma-2}{3 \gamma-4}}$, which are compatible with the results in subsection \ref{subsection3.1-Scaling relations between solutions to partial differential equations} as one can expect.

\begin{theorem}[Relations between Minimizers and Minimal Energies with Different Mass] \label{scaling relation for energy}
	$E_0(\rho)$ allows a unique minimizer on $m {R}\left(\mathbb{R}^3\right)$ up to translation. Let $\sigma$ be the minimizer with mass 1 and the corresponding minimal energy is $e_0=E_0(\sigma)$, then any other minimizer with mass $m$ can be represented as $\sigma_m(x)=\frac{1}{A} \sigma\left(\frac{1}{B} x\right)$, where $A=m^{-\frac{2}{3 \gamma-4}}, B=m^{\frac{\gamma-2}{3 \gamma-4}}$, and the corresponding minimal energy is $e_0(m)=E_0\left(\sigma_m\right)=m^{\frac{5 \gamma-6}{3 \gamma-4}} e_0$.
\end{theorem}

\begin{proof}
	{
		\rm
		The uniqueness result comes from Proposition \ref{uniqueness in non-rotating cases}. The relation between $\sigma_m$ and $\sigma$ can be seen in the arguments above. Notice $A^{-\gamma} B^3=m^{\frac{5 \gamma-6}{3 \gamma-4}}$, thus we have $e_0(m)=m^{\frac{5 \gamma-6}{3 \gamma-4}} e_0$.
	}
\end{proof}

\vspace*{0.8 em}
\begin{remark}\label{asymptotic behaviour of radius and norm in single star case}
	Thanks to Theorem \ref{non-rotating}, we know $\|\sigma\|_{L^{\infty}\left(\mathbb{R}^3\right)} \leq C_1$ and spt $\sigma$ is contained in a ball of radius $R_1$, therefore, $\left\|\sigma_m\right\|_{L^{\infty}\left(\mathbb{R}^3\right)} \leq C_m=\frac{C_1}{A}$ and spt $\sigma_m$ is contained in a ball of radius $R_m=B R_1$. Since $A=m^{-\frac{2}{3 \gamma-4}}, B=m^{\frac{\gamma-2}{3 \gamma-4}}$, if we further assume $\gamma >2$, we know $\lim\limits_{m \rightarrow 0} A=+\infty$ and $\lim\limits_{m \rightarrow 0} B=0$, thus $\left\|\sigma_m\right\|_{L^{\infty}\left(\mathbb{R}^3\right)}$ and the size of $\sigma_m$'s support will go to 0 when $m \rightarrow 0$, with rates $m^{\frac{2}{3 \gamma-4}}$ and $m^{\frac{\gamma-2}{3 \gamma-4}}$ respectively. Notice $\left\|\sigma_m\right\|_{L^{\infty}\left(\mathbb{R}^3\right)}$ is actually the central density of $\sigma_m$, thus it gives the decay rate and strengthens Corollary \ref{relation between central density and mass strenthed}. However, when $\gamma<2$, $\lim\limits_{m \rightarrow 0} B=\infty$, the size of $\sigma_m$'s support can go to $\infty$ when $m \rightarrow 0$ (flatten out or dispread). It coincides with a result mentioned in Lieb and Yau's paper \cite[Theorem 5]{LY87}, which says that the radius $R_m \rightarrow \infty$ as $m \rightarrow 0$. In their paper quantum mechanics (fermions case) is discussed.
\end{remark}

Since we have the scaling relationship between energies, intuitively, one might guess whether a scaling relationship exists between the derivatives. Indeed, there is a relationship.

\begin{proposition}[Relations of Variational Derivatives Between Densities with Different Mass]\label{scaling relation for derivatives}
	Given $\rho_m \in mR(\mathbb{R}^3)$ with $U(\rho_m)<\infty$, then $E_0\left(\rho_m\right)$ is $P_{\infty}\left(\rho_m\right)$ differentiable at $\rho_m$ and the derivative at $\rho_m$ satisfies $$E_0^{\prime}\left(\rho_m\right)(x)= A^{-\gamma-1}E_0^{\prime}(\rho)\left(\frac{x}{B}\right)=m^{\frac{2 \gamma-2}{3 \gamma-4}}E_0^{\prime}(\rho)\left(\frac{1}{B} x\right)$$
	where $\rho(x)\coloneq\left(\rho_m\right)_{A, B}(x)=A \rho_m(B x) \in R(\mathbb{R}^3)$ with $U(\rho)<\infty$, $A=m^{-\frac{2}{3 \gamma-4}}$, $B=m^{\frac{\gamma-2}{3 \gamma-4}}$, $\rho$ has mass 1.
\end{proposition}

\begin{proof}
	{
		\rm
		The result $E_0\left(\rho_m\right)$ is $P_{\infty}\left(\rho_m\right)$-differentiable at $\rho_m$ comes from Lemma \ref{diff. of energy}. Also easy to check $U(\rho)<\infty$ and has mass 1, and $E_0(\rho)$ is $P_{\infty}(\rho)$-differentiable at $\rho$. Given $\tau_m \in P_{\infty}\left(\rho_m\right)$, one can check $\tau:=\left(\tau_m\right)_{A, B}$ is in $P_{\infty}(\rho)$. By the definition of variational derivatives, one has
		\begin{align*}
			\int_{\mathbb{R}^3} E_0^{\prime}\left(\rho_m\right) \tau_m\,dx & =\lim _{t \rightarrow \infty} \frac{E_0\left(\rho_m+t \tau_m\right)-E_0\left(\rho_m\right)}{t}\\
			&=\lim _{t \rightarrow \infty} \frac{E_0\left((\rho+t \tau)_{\frac{1}{A},\frac{1}{B}}\right)-E_0\left(\rho_{\frac{1}{A},\frac{1}{B}}\right)}{t} \\
			& =\lim _{t \rightarrow \infty} A^{-\gamma} B^3 \frac{E_0(\rho+t \tau)-E_0(\rho)}{t}\\
			&=A^{-\gamma} B^3 \int_{\mathbb{R}^3} E_0^{\prime}(\rho) \tau\,dx \\
			& =A^{-\gamma-1} \int_{\mathbb{R}^3} E_0^{\prime}(\rho)(x) \tau_m(B x) \,d(B x)\\
			&=A^{-\gamma-1} \int_{\mathbb{R}^3} E_0^{\prime}(\rho)\left(\frac{x}{B}\right) \tau_m(x) \,dx
		\end{align*}

		Since there are enough functions $\tau_m$ in $P_{\infty}\left(\rho_m\right)$, thus we have $$E_0^{\prime}\left(\rho_m\right)(x)= A^{-\gamma-1}E_0^{\prime}(\rho)\left(\frac{x}{B}\right)=m^{\frac{2 \gamma-2}{3 \gamma-4}} E_0^{\prime}(\rho)\left(\frac{1}{B} x\right)$$
	}
\end{proof}

\begin{remark}
	Suppose $U(\rho_m)<\infty $, there is another way to show the result. Let $\rho(x)=A \rho_m(B x)$ with $A=m^{-\frac{2}{3 \gamma-4}}$ and $B=m^{\frac{\gamma-2}{3 \gamma-4}}$, we already know $E_0^{\prime}(\rho)=A^{\prime}(\rho(x))- V_\rho(x)$, and then one can also compute explicitly that $$A^{\prime}\left(\rho_m\right)(x)=A^{\prime}\left(\frac{1}{A} \rho\right)\left(\frac{1}{B} x\right)=\frac{1}{A^{\gamma-1}} A^{\prime}(\rho)\left(\frac{1}{B} x\right)$$ $$V_{\rho_m}(x)=\frac{B^2}{A} V_\rho\left(\frac{1}{B} x\right)=\frac{1}{A^{\gamma-1}} V_\rho\left(\frac{1}{B} x\right)$$
	and obtain again 
	$$E_0^{\prime}\left(\rho_m\right)(x)=\frac{1}{A^{\gamma-1}} E_0^{\prime}(\rho)\left(\frac{1}{B} x\right)=m^{\frac{2 \gamma-2}{3 \gamma-4}} E_0^{\prime}(\rho)\left(\frac{1}{B} x\right)$$
\end{remark}

	\section*{Appendix}\label{Appendix}
	\appendix

	\section{Properties of Sobolev Spaces}\label{sectionA-properties of sobolev spaces}
	
	We recall some properties of Sobolev spaces as well as of $L^p$ spaces. Some of them will be used frequently in this paper.

	
	\begin{proposition}[Hardy-Littlewood-Sobolev Inequality {\cite[Theorem 1.7]{BCD11}}] \label{HLSI}
		Let $1<p, r<\infty$ and $0<\alpha<n$ be such that $\frac{1}{p}+\frac{\alpha}{n}=\frac{1}{r}+1$. $\exists C_{p, \alpha, n}>0$, such that
		\begin{equation} \label{BoV}
			\left\||\cdot|^{-\alpha} * f\right\|_{L^{r}\left(\mathbb{R}^{n}\right)} \leq C_{p, \alpha, n}\|f\|_{L^{p}\left(\mathbb{R}^{n}\right)}
		\end{equation}
	\end{proposition}

	By Hardy-Littlewood-Sobolev Inequality Proposition \ref{HLSI}, we can show the potential energy is bounded by density's $L^{1}$ norm (density's mass) and density's $L^{\frac{4}{3}}$ norm.
	
	\begin{proposition}[Bound of Potential Energy {\cite[Proposition 6]{AB71}}]\label{bound of potential energy}
		$\exists C>0$, if $\rho \in L^{1}\left(\mathbb{R}^{3}\right) \cap L^{\frac{4}{3}}\left(\mathbb{R}^{3}\right)$, let $V_{\rho}(x)=$ $\int_{\mathbb{R}^3} \frac{\rho(y)}{|x-y|} \,dy$, then 
		$$\left|\int_{\mathbb{R}^3} \rho \cdot V_{\rho}\,dx\right| \leq C \int_{\mathbb{R}^3}|\rho|^{\frac{4}{3}}\,dx \cdot\left(\int_{\mathbb{R}^3}|\rho|\,dx\right)^{\frac{2}{3}}$$
	\end{proposition}

	Notice that Hardy-Littlewood-Sobolev Inequality Proposition \ref{HLSI} can fail when $r=\infty$. To estimate the bound in $L^{\infty}$, we introduce the following proposition:
	
	\begin{proposition}[Bound of Potential {\cite[Proposition 5]{AB71}} {\cite[Proposition B.1]{Che26G1}}] \label{bound of potential}
		Suppose $\rho \in L^{1}\left(\mathbb{R}^{3}\right) \cap L^{p}\left(\mathbb{R}^{3}\right)$. If $1<p \leq \frac{3}{2}$, then $\forall r \in\left(3, \frac{3 p}{3-2 p}\right)$, we have $V_{\rho} \in L^{r}\left(\mathbb{R}^{3}\right)$, and $\exists 0<b_{r}<1,0<c_{r}<1, C>0$, such that
		\begin{equation} \label{BoVs}
			\left\|V_{\rho}\right\|_{L^{r}} \leq C\left(\|\rho\|_{L^{1}}^{b_{r}}\|\rho\|_{L^{p}}^{1-b_{r}}+\|\rho\|_{L^{1}}^{c_{r}}\|\rho\|_{L^{p}}^{1-c_{r}}\right) 
		\end{equation}
		
		If $p>\frac{3}{2}$, then $V_{\rho}$ is bounded and continuous and satisfies (\ref{BoVs}) with $r=\infty$.
	\end{proposition}

	As discussed in \cite[Appendix B]{Che26G1}, it turns out $h(x)=-\int_{\mathbb{R}^3} \frac{y_{j} \rho(x-y)}{|y|^{3}} \,dy$ is the weak derivative of $V_{\rho}(x)$. Moreover, for $p>3$, we have $\frac{\partial V_{\rho}}{\partial x_{j}}(x)=-\int_{\mathbb{R}^3} \frac{y_{j} \rho(x-y)}{|y|^{3}} \,dy$ not only in the distribution sense but also in the classical sense:
	
	\begin{proposition}[Differentiability of Potential {\cite[Section 3]{AB71}} {\cite[Proposition B.3]{Che26G1}}] \label{diff. of poten.}
		If $\rho \in L^{1}\left(\mathbb{R}^{3}\right) \cap L^{p}\left(\mathbb{R}^{3}\right)$ for some $p>3$, then $V_{\rho} \in$ $W^{1, \infty}\left(\mathbb{R}^{3}\right)$ is continuously differentiable, and the weak derivative coincides with the classical one for all $x \in \mathbb{R}^{3}$.
	\end{proposition}
	
To prove the low semicontinuity of the energy, we hope to connect weak convergence of densities and strong convergence of potentials.

\begin{proposition}[Compactness of Convolution (Potential) Operator with $3<q<\infty$]\label{compactness of convolution for q>3}
	Let $G(x)=\frac{1}{|x|}$, then $G \in L_{w}^{3}\left(\mathbb{R}^{N}\right)$. Let $F=G * B$, where $B$ is a bounded set in $L^{q}\left(\mathbb{R}^{N}\right)$ with $3<q<\infty$. Let $\Omega \subset \mathbb{R}^{3}$ be open, bounded and of class $C^{1}$, then $\left.F\right|_{\Omega}$ has compact closure in $C(\bar{\Omega})$. In particular, let $V_{\rho}(x)=\int_{\mathbb{R}^3} \frac{\rho(y)}{|x-y|} \,dy$, then for all $r\in[1,\infty]$, the map $\mathcal{V}: \rho \rightarrow V_{\rho}$ from $L^{q}(\Omega)$ to $L^{r}(\Omega)$ is compact.
\end{proposition}

\begin{proof}
	{
		\rm
		By Proposition \ref{diff. of poten.} we know $F \subset W^{1, \infty}\left(\mathbb{R}^{3}\right)$. Following its proof in \cite{Che26G1} in detail one can see the norm $\|\cdot\|_{W^{1, \infty}}$ is bounded uniformly. The result comes from Rellich-Kondrachov Theorem \cite[Theorem 9.16]{Bre11}.
	}
\end{proof}

Actually, one can have a more general version:
\begin{proposition}[Compactness of Convolution (Potential) Operator with $q \geq 1$] \label{compactness of convolution for q>=1}
	Let $G(x)=\frac{1}{|x|}$, then $G \in L_{w}^{3}\left(\mathbb{R}^{3}\right)$. Let $F=G * B$, where $B$ is bounded both in $L^{1}\left(\mathbb{R}^{3}\right)$ and $L^{q}\left(\mathbb{R}^{3}\right)$ with $q > 1$. Let $\Omega \subset \mathbb{R}^{3}$ be open, bounded and of class $C^{1}$, then
	\begin{enumerate}[(1)]
		\item if $1 < q<\frac{3}{2}$, then $\forall 1 \leq r<p$, where $\frac{1}{p}=\frac{1}{q}-\frac{2}{3},\left.F\right|_{\Omega}$ has compact closure in $L^{r}(\Omega)$;
		\item if $q=\frac{3}{2}$, then $\forall 1 \leq r<\infty,\left.F\right|_{\Omega}$ has compact closure in $L^{r}(\Omega)$;
		\item if $q>\frac{3}{2}$, then $\left.F\right|_{\Omega}$ has compact closure in $C(\bar{\Omega})$.
	\end{enumerate}
\end{proposition}

\begin{proof}
	{
		We first establish the boundedness of $F$ in $W^{1, s}$ for some appropriate $s$. Similar to the proof of \cite[Proposition B.1]{Che26G1} (cf. Proposition \ref{bound of potential} above), and noting that $K(y) = \frac{y_j}{|y|^3}$ is in $L^p_{\text{loc}}(\mathbb{R}^3)$ for $p < 3/2$ and belongs to $L^p(\mathbb{R}^3 \setminus B_R(0))$ for $p > 3/2$, we can show there exists an appropriate $s>0$, such that $\forall \rho \in B$, $V_\rho(x)$ and $\frac{\partial V_{\rho}}{\partial x_{j}}(x)=-\int_{\mathbb{R}^3} \frac{y_{j} \rho(x-y)}{|y|^{3}} \,dy$ are bounded in $L^s$. Since $B$ is bounded in $L^1$ and $L^q$, the bound of $V_\rho$ in $W^{1, s}$ is uniform for all $\rho \in B$. Once the boundedness of $F$ in $W^{1, s}$ is established, we apply the Rellich–Kondrachov theorem \cite[Theorem 9.16]{Bre11} to conclude the proof.

	}
\end{proof}

%

\section*{Acknowledgments}\label{section-acknowledgments}
\addcontentsline{toc}{section}{Acknowledgments}
The author is partially supported by the National Science Foundation grant DMS-2308208. This work was primarily carried out during the author’s Master’s studies at the University of Bonn. The author thanks Juan Velázquez and Dimitri Cobb for their continued advice and support since the author’s time in Bonn, as well as Christof Sparber and Mimi Dai for their comments and support during the author's Ph.D. studies at the University of Illinois Chicago. Thanks also to Lorenzo Pompili, Shao Liu, Xiaopeng Cheng, Bernhard Kepka, Daniel Sánchez Simón del Pino for discussions, and to Théophile Dolmaire and other instructors. The author is grateful to his parents.

\bibliographystyle{abbrv}
\addcontentsline{toc}{section}{References}
\bibliography{references}

@misc{Che26G1,
	title={Gradient Existence and Energy Finiteness of Local Minimizers in the {W}asserstein $L^\infty$ Topology for Binary-Star Systems}, 
	author={Hangsheng Chen},
	year={2026},
	eprint={2602.01678},
	archivePrefix={arXiv},
	primaryClass={math.AP},
	url={https://arxiv.org/abs/2602.01678}, 
}

@misc{Che26E3,
	title={Existence for Stable Rotating Star-Planet Systems}, 
	author={Hangsheng Chen},
	year={2026},
	eprint={2602.02761},
	archivePrefix={arXiv},
	primaryClass={math.AP},
	url={https://arxiv.org/abs/2602.02761}, 
}

@misc{CVD24,
	language = {eng},
	publisher = {Rheinische Friedrich-Wilhelms-Universität Bonn},
	title = {Existence for stable rotating star-planet systems},
	address = {Bonn},
	author = {Chen, Hangsheng and Velázquez, J. J. L. and Cobb, Dimitri and Rheinische Friedrich-Wilhelms-Universität Bonn Begründer eines Werks},
	year = {2024},
}

@Article{Sob38,
	title={ On a theorem of functional analysis },
	author={ S. L. Sobolev },
	journal={ Matematicheskii Sbornik },
	year={ 1938 },	
	volume={ 4 },
	pages={ 471-497 },
}

@Article{CDD10,
	author={Campos, Juan
	and del Pino, Manuel
	and Dolbeault, Jean},
	title={Relative Equilibria in Continuous Stellar Dynamics},
	journal={Communications in Mathematical Physics},
	year={2010},
	month={Dec},
	day={01},
	volume={300},
	number={3},
	pages={765-788},
	issn={1432-0916},
	doi={10.1007/s00220-010-1128-2},
	url={https://doi.org/10.1007/s00220-010-1128-2}
}

@Article{HRS22,
	author={Had{\v{z}}i{\'{c}}, Mahir
	and Rein, Gerhard
	and Straub, Christopher},
	title={On the Existence of Linearly Oscillating Galaxies},
	journal={Archive for Rational Mechanics and Analysis},
	year={2022},
	month={Feb},
	day={01},
	volume={243},
	number={2},
	pages={611-696},
	issn={1432-0673},
	doi={10.1007/s00205-021-01734-4},
	url={https://doi.org/10.1007/s00205-021-01734-4}
}

@article{Wol99,
	title = {On nonlinear stability of polytropic galaxies},
	journal = {Annales de l'Institut Henri Poincaré C, Analyse non linéaire},
	volume = {16},
	number = {1},
	pages = {15-48},
	year = {1999},
	issn = {0294-1449},
	doi = {https://doi.org/10.1016/S0294-1449(99)80007-9},
	url = {https://www.sciencedirect.com/science/article/pii/S0294144999800079},
	author = {G. Wolansky}
}

@Article{LMR12,
	author={Lemou, Mohammed
	and M{\'e}hats, Florian
	and Rapha{\"e}l, Pierre},
	title={Orbital stability of spherical galactic models},
	journal={Inventiones mathematicae},
	year={2012},
	month={Jan},
	day={01},
	volume={187},
	number={1},
	pages={145-194},
	issn={1432-1297},
	doi={10.1007/s00222-011-0332-9},
	url={https://doi.org/10.1007/s00222-011-0332-9}
}

@Article{LMR11,
	author={Lemou, Mohammed
	and M{\'e}hats, Florian
	and Rapha{\"e}l, Pierre},
	title={A New Variational Approach to the Stability of Gravitational Systems},
	journal={Communications in Mathematical Physics},
	year={2011},
	month={Feb},
	day={01},
	volume={302},
	number={1},
	pages={161-224},
	issn={1432-0916},
	doi={10.1007/s00220-010-1182-9},
	url={https://doi.org/10.1007/s00220-010-1182-9}
}

@Article{LMR08,
	author={Lemou, Mohammed
	and M{\'e}hats, Florian
	and Raphael, Pierre},
	title={The Orbital Stability of the Ground States and the Singularity Formation for the Gravitational Vlasov Poisson System},
	journal={Archive for Rational Mechanics and Analysis},
	year={2008},
	month={Sep},
	day={01},
	volume={189},
	number={3},
	pages={425-468},
	issn={1432-0673},
	doi={10.1007/s00205-008-0126-4},
	url={https://doi.org/10.1007/s00205-008-0126-4}
}

@article{LMR05,
	title = {Orbital stability and singularity formation for Vlasov–Poisson systems},
	journal = {Comptes Rendus Mathematique},
	volume = {341},
	number = {4},
	pages = {269-274},
	year = {2005},
	issn = {1631-073X},
	doi = {https://doi.org/10.1016/j.crma.2005.06.018},
	url = {https://www.sciencedirect.com/science/article/pii/S1631073X0500275X},
	author = {Mohammed Lemou and Florian Méhats and Pierre Raphael}
}

@Article{GR07,
	author={Guo, Yan
	and Rein, Gerhard},
	title={A Non-Variational Approach to Nonlinear Stability in Stellar Dynamics Applied to the King Model},
	journal={Communications in Mathematical Physics},
	year={2007},
	month={Apr},
	day={01},
	volume={271},
	number={2},
	pages={489-509},
	issn={1432-0916},
	doi={10.1007/s00220-007-0212-8},
	url={https://doi.org/10.1007/s00220-007-0212-8}
}

@article{GR03,
	author = {Guo, Y. and Rein, G.},
	title = "{Stable models of elliptical galaxies}",
	journal = {Monthly Notices of the Royal Astronomical Society},
	volume = {344},
	number = {4},
	pages = {1296-1306},
	year = {2003},
	month = {10},
	issn = {0035-8711},
	doi = {10.1046/j.1365-8711.2003.06920.x},
	url = {https://doi.org/10.1046/j.1365-8711.2003.06920.x},
	eprint = {https://academic.oup.com/mnras/article-pdf/344/4/1296/2918489/344-4-1296.pdf},
}

@Article{GR99S,
	author={Guo, Yan
	and Rein, Gerhard},
	title={Stable Steady States in Stellar Dynamics},
	journal={Archive for Rational Mechanics and Analysis},
	year={1999},
	month={Aug},
	day={01},
	volume={147},
	number={3},
	pages={225-243},
	issn={1432-0673},
	doi={10.1007/s002050050150},
	url={https://doi.org/10.1007/s002050050150}
}

@article{GR99E,
	ISSN = {00222518, 19435258},
	URL = {http://www.jstor.org/stable/24901014},
	author = {Yan Guo and Gerhard Rein},
	journal = {Indiana University Mathematics Journal},
	number = {4},
	pages = {1237--1255},
	publisher = {Indiana University Mathematics Department},
	title = {Existence and Stability of Camm Type Steady States in Galactic Dynamics},
	urldate = {2024-05-13},
	volume = {48},
	year = {1999}
}

@Article{GL08,
	author={Guo, Yan
	and Li, Zhiwu},
	title={Unstable and Stable Galaxy Models},
	journal={Communications in Mathematical Physics},
	year={2008},
	month={May},
	day={01},
	volume={279},
	number={3},
	pages={789-813},
	issn={1432-0916},
	doi={10.1007/s00220-008-0439-z},
	url={https://doi.org/10.1007/s00220-008-0439-z}
}

@Article{Guo99,
	author={Guo, Yan},
	title={Variational Method{\textparagraph}for Stable Polytropic Galaxies},
	journal={Archive for Rational Mechanics and Analysis},
	year={1999},
	month={Dec},
	day={01},
	volume={150},
	number={3},
	pages={209-224},
	issn={1432-0673},
	doi={10.1007/s002050050187},
	url={https://doi.org/10.1007/s002050050187}
}

@incollection{Rei07,
	title = {Chapter 5 - Collisionless Kinetic Equations from Astrophysics – The Vlasov–Poisson System},
	editor = {C.M. Dafermos and E. Feireisl},
	series = {Handbook of Differential Equations: Evolutionary Equations},
	publisher = {North-Holland},
	volume = {3},
	pages = {383-476},
	year = {2007},
	issn = {1874-5717},
	doi = {https://doi.org/10.1016/S1874-5717(07)80008-9},
	url = {https://www.sciencedirect.com/science/article/pii/S1874571707800089},
	author = {Gerhard Rein}
}

@Article{GR01,
	author={Guo, Yan
	and Rein, Gerhard},
	title={Isotropic Steady States in Galactic Dynamics},
	journal={Communications in Mathematical Physics},
	year={2001},
	month={Jun},
	day={01},
	volume={219},
	number={3},
	pages={607-629},
	issn={1432-0916},
	doi={10.1007/s002200100434},
	url={https://doi.org/10.1007/s002200100434}
}

@BOOK{BT87,
	author = {{Binney}, James and {Tremaine}, Scott},
	title = "{Galactic dynamics}",
	year = 1987,
	adsurl = {https://ui.adsabs.harvard.edu/abs/1987gady.book.....B}
}

@Article{BFH86,
	author={Batt, J.
	and Faltenbacher, W.
	and Horst, E.},
	title={Stationary spherically symmetric models in stellar dynamics},
	journal={Archive for Rational Mechanics and Analysis},
	year={1986},
	month={Jun},
	day={01},
	volume={93},
	number={2},
	pages={159-183},
	issn={1432-0673},
	doi={10.1007/BF00279958},
	url={https://doi.org/10.1007/BF00279958}
}

@article{Sch08,
	author = {Schulze, Achim},
	year = {2008},
	month = {09},
	pages = {711-727},
	title = {Existence of axially symmetric solutions to the Vlasov-Poisson system depending on Jacobi's integral},
	volume = {6},
	journal = {Communications in Mathematical Sciences}
}

@misc{AKV23,
	title={Rotating solutions to the incompressible Euler-Poisson equation with external particle}, 
	author={Diego Alonso-Orán and Bernhard Kepka and Juan J. L. Velázquez},
	year={2023},
	eprint={2302.01146},
	archivePrefix={arXiv},
	primaryClass={math.AP}
}

@Article{Rei03,
	author={Rein, Gerhard},
	title={Non-Linear Stability of Gaseous Stars},
	journal={Archive for Rational Mechanics and Analysis},
	year={2003},
	month={Jun},
	day={01},
	volume={168},
	number={2},
	pages={115-130},
	issn={1432-0673},
	doi={10.1007/s00205-003-0260-y},
	url={https://doi.org/10.1007/s00205-003-0260-y}
}

@article{Jan14,
	author = {Jang, Juhi},
	title = {Nonlinear Instability Theory of Lane-Emden Stars},
	journal = {Communications on Pure and Applied Mathematics},
	volume = {67},
	number = {9},
	pages = {1418-1465},
	doi = {https://doi.org/10.1002/cpa.21499},
	url = {https://onlinelibrary.wiley.com/doi/abs/10.1002/cpa.21499},
	eprint = {https://onlinelibrary.wiley.com/doi/pdf/10.1002/cpa.21499},
	year = {2014}
}

@Article{DLYY02,
	author={Deng, Yinbin
	and Liu, Tai-Ping
	and Yang, Tong
	and Yao, Zheng-an},
	title={Solutions of Euler-Poisson Equations{\textparagraph}for Gaseous Stars},
	journal={Archive for Rational Mechanics and Analysis},
	year={2002},
	month={Sep},
	day={01},
	volume={164},
	number={3},
	pages={261-285},
	issn={1432-0673},
	doi={10.1007/s00205-002-0209-6},
	url={https://doi.org/10.1007/s00205-002-0209-6}
}

@Article{LS09,
	author={Luo, Tao
	and Smoller, Joel},
	title={Existence and Non-linear Stability of Rotating Star Solutions of the Compressible Euler--Poisson Equations},
	journal={Archive for Rational Mechanics and Analysis},
	year={2009},
	month={Mar},
	day={01},
	volume={191},
	number={3},
	pages={447-496},
	issn={1432-0673},
	doi={10.1007/s00205-007-0108-y},
	url={https://doi.org/10.1007/s00205-007-0108-y}
}

@Article{LS08,
	author={Luo, Tao
	and Smoller, Joel},
	title={Nonlinear Dynamical Stability of Newtonian Rotating and Non-rotating White Dwarfs and Rotating Supermassive Stars},
	journal={Communications in Mathematical Physics},
	year={2008},
	month={Dec},
	day={01},
	volume={284},
	number={2},
	pages={425-457},
	issn={1432-0916},
	doi={10.1007/s00220-008-0569-3},
	url={https://doi.org/10.1007/s00220-008-0569-3}
}

@Article{SW19,
	author={Strauss, Walter A.
	and Wu, Yilun},
	title={Rapidly Rotating Stars},
	journal={Communications in Mathematical Physics},
	year={2019},
	month={Jun},
	day={01},
	volume={368},
	number={2},
	pages={701-721},
	issn={1432-0916},
	doi={10.1007/s00220-019-03414-7},
	url={https://doi.org/10.1007/s00220-019-03414-7}
}

@article{SW17,
	author = {Strauss, Walter A. and Wu, Yilun},
	title = {Steady States of Rotating Stars and Galaxies},
	journal = {SIAM Journal on Mathematical Analysis},
	volume = {49},
	number = {6},
	pages = {4865-4914},
	year = {2017},
	doi = {10.1137/17M1119391},
	URL = { 
	https://doi.org/10.1137/17M1119391
	},
	eprint = { 
	https://doi.org/10.1137/17M1119391
	}
}

@article{Hei94,
	author = {Heilig, U.},
	title = {On {Lichtenstein's} analysis of rotating newtonian stars},
	journal = {Annales de l'I.H.P. Physique th\'eorique},
	pages = {457--487},
	publisher = {Gauthier-Villars},
	volume = {60},
	number = {4},
	year = {1994},
	mrnumber = {1288588},
	zbl = {0808.35107},
	language = {en},
	url = {http://www.numdam.org/item/AIHPA_1994__60_4_457_0/}
}

@Article{JM17,
	author={Jang, Juhi
	and Makino, Tetu},
	title={On Slowly Rotating Axisymmetric Solutions of the Euler--Poisson Equations},
	journal={Archive for Rational Mechanics and Analysis},
	year={2017},
	month={Aug},
	day={01},
	volume={225},
	number={2},
	pages={873-900},
	issn={1432-0673},
	doi={10.1007/s00205-017-1115-2},
	url={https://doi.org/10.1007/s00205-017-1115-2}
}

@Article{CL94,
	author={Chanillo, Sagun
	and Li, Yan Yan},
	title={On diameters of uniformly rotating stars},
	journal={Communications in Mathematical Physics},
	year={1994},
	month={Dec},
	day={01},
	volume={166},
	number={2},
	pages={417-430},
	issn={1432-0916},
	doi={10.1007/BF02112323},
	url={https://doi.org/10.1007/BF02112323}
}

@article{CF80,
	title = {The shape of axisymmetric rotating fluid},
	journal = {Journal of Functional Analysis},
	volume = {35},
	number = {1},
	pages = {109-142},
	year = {1980},
	issn = {0022-1236},
	doi = {https://doi.org/10.1016/0022-1236(80)90082-8},
	url = {https://www.sciencedirect.com/science/article/pii/0022123680900828},
	author = {Luis A Caffarelli and Avner Friedman}
}

@Article{Auc91,
	author={Auchmuty, Giles},
	title={The global branching of rotating stars},
	journal={Archive for Rational Mechanics and Analysis},
	year={1991},
	month={Jun},
	day={01},
	volume={114},
	number={2},
	pages={179-193},
	issn={1432-0673},
	doi={10.1007/BF00375402},
	url={https://doi.org/10.1007/BF00375402}
}

@article{AB71M,
	author = {J. F. G. Auchmuty and Richard Beals},
	year = {1971},
	month = {04},
	pages = {L79},
	title = {Models of Rotating Stars},
	volume = {165},
	journal = {The Astrophysical Journal},
	doi = {10.1086/180721}
}

@book{Jar13,
	title={Theories of Figures of Celestial Bodies},
	author={Jardetzky, W.S.},
	isbn={9780486174662},
	series={Dover Books on Physics},
	url={https://books.google.de/books?id=ugnCAgAAQBAJ},
	year={2013},
	publisher={Dover Publications}
}

@article{Cha67,
	author = {Chandrasekhar, S.},
	title = {Ellipsoidal figures of equilibrium—an historical account},
	journal = {Communications on Pure and Applied Mathematics},
	volume = {20},
	number = {2},
	pages = {251-265},
	doi = {https://doi.org/10.1002/cpa.3160200203},
	url = {https://onlinelibrary.wiley.com/doi/abs/10.1002/cpa.3160200203},
	eprint = {https://onlinelibrary.wiley.com/doi/pdf/10.1002/cpa.3160200203},
	year = {1967}
}

@Article{Lic33,
	author={Lichtenstein, Leon},
	title={Untersuchungen {\"u}ber die Gleichgewichtsfiguren rotierender Fl{\"u}ssigkeiten, deren Teilchen einander nach dem Newtonschen Gesetze anziehen},
	journal={Mathematische Zeitschrift},
	year={1933},
	month={Dec},
	day={01},
	volume={36},
	number={1},
	pages={481-562},
	issn={1432-1823},
	doi={10.1007/BF01188634},
	url={https://doi.org/10.1007/BF01188634}
}

@article{Cha33,
	author = {Chandrasekhar, S. and Milne, E. A.},
	title = "{The Equilibrium of Distorted Polytropes: (I). The Rotational Problem}",
	journal = {Monthly Notices of the Royal Astronomical Society},
	volume = {93},
	number = {5},
	pages = {390-406},
	year = {1933},
	month = {03},
	issn = {0035-8711},
	doi = {10.1093/mnras/93.5.390},
	url = {https://doi.org/10.1093/mnras/93.5.390},
	eprint = {https://academic.oup.com/mnras/article-pdf/93/5/390/2793671/mnras93-0390.pdf},
}

@article{Zei24,
	author = {Zeipel, H. v.},
	title = "{The Radiative Equilibrium of a Slightly Oblate Rotating Star}",
	journal = {Monthly Notices of the Royal Astronomical Society},
	volume = {84},
	number = {9},
	pages = {684-702},
	year = {1924},
	month = {07},
	issn = {0035-8711},
	doi = {10.1093/mnras/84.9.684},
	url = {https://doi.org/10.1093/mnras/84.9.684},
	eprint = {https://academic.oup.com/mnras/article-pdf/84/9/684/2793260/mnras84-0684.pdf},
}

@article{Mil23,
	author = {Milne., E. A.},
	title = "{The Equilibrium of a Rotating Star}",
	journal = {Monthly Notices of the Royal Astronomical Society},
	volume = {83},
	number = {3},
	pages = {118-147},
	year = {1923},
	month = {01},
	issn = {0035-8711},
	doi = {10.1093/mnras/83.3.118},
	url = {https://doi.org/10.1093/mnras/83.3.118},
	eprint = {https://academic.oup.com/mnras/article-pdf/83/3/118/3420562/mnras83-0118.pdf},
}

@book{Cha39,
	title={An Introduction to the Study of Stellar Structure},
	author={Chandrasekhar, S.},
	lccn={39008320},
	series={Astrophysical monographs},
	url={https://books.google.de/books?id=5duRzgEACAAJ},
	year={1939},
	publisher={University of Chicago Press}
}

@book{Rud76,
	title={Principles of Mathematical Analysis},
	author={Rudin, W.},
	isbn={9780070856134},
	lccn={75179033},
	series={International series in pure and applied mathematics},
	url={https://books.google.de/books?id=kwqzPAAACAAJ},
	year={1976},
	publisher={McGraw-Hill}
}

@article{Cob24,
	title = {Bounded solutions in incompressible hydrodynamics},
	journal = {Journal of Functional Analysis},
	volume = {286},
	number = {5},
	pages = {110290},
	year = {2024},
	issn = {0022-1236},
	doi = {https://doi.org/10.1016/j.jfa.2023.110290},
	url = {https://www.sciencedirect.com/science/article/pii/S0022123623004470},
	author = {Dimitri Cobb},
}

@book{NCW99,
	title={The Principia: Mathematical Principles of Natural Philosophy},
	author={Newton, I. and Cohen, I.B. and Whitman, A.},
	isbn={9780520088160},
	lccn={99010278},
	series={The Principia: Mathematical Principles of Natural Philosophy},
	url={https://books.google.de/books?id=k_NgQgAACAAJ},
	year={1999},
	publisher={University of California Press}
}

@book{New87,
	title={Philosophiae naturalis principia mathematica},
	author={Newton, I.},
	series={Early English books online},
	url={https://books.google.de/books?id=-dVKAQAAIAAJ},
	year={1687},
	publisher={Jussu Societas Regi{\ae} ac typis Josephi Streater, prostant venales apud Sam. Smith}
}

@misc{She24,
	author = "{Wikipedia contributors}",
	title = "Shell theorem --- {Wikipedia}{,} The Free Encyclopedia",
	year = "2024",
	url = "https://en.wikipedia.org/w/index.php?title=Shell_theorem&oldid=1212159714",
	note = "[Online; accessed 21-April-2024]"
}

@book{Eva10,
	title={Partial Differential Equations},
	author={Evans, L.C.},
	isbn={9780821849743},
	lccn={2009044716},
	series={Graduate studies in mathematics},
	url={https://books.google.de/books?id=Xnu0o_EJrCQC},
	year={2010},
	publisher={American Mathematical Society}
}

@book{Dal93,
	title={An introduction to $\Gamma$-convergence},
	author={Dal Maso, Gianni},
	volume={8},
	year={1993},
	series={Progress in Nonlinear Differential Equations and Their Applications},
	publisher={Birkhäuser Boston, MA}
}

@article{Lie77,
	author = {Lieb, Elliott H.},
	title = {Existence and Uniqueness of the Minimizing Solution of {C}hoquard's Nonlinear Equation},
	journal = {Studies in Applied Mathematics},
	volume = {57},
	number = {2},
	pages = {93-105},
	doi = {https://doi.org/10.1002/sapm197757293},
	url = {https://onlinelibrary.wiley.com/doi/abs/10.1002/sapm197757293},
	eprint = {https://onlinelibrary.wiley.com/doi/pdf/10.1002/sapm197757293},
	year = {1977}
}

@book{Bre11,
	title={Functional analysis, Sobolev spaces and partial differential equations},
	author={Brezis, Haim},
	volume={2},
	number={3},
	year={2011},
	publisher={Springer}
}

@book{BCD11,
	title={Fourier Analysis and Nonlinear Partial Differential Equations},
	author={Bahouri, H. and Chemin, J.Y. and Danchin, R.},
	isbn={9783642168307},
	series={Grundlehren der mathematischen Wissenschaften},
	url={https://books.google.de/books?id=CcTnaveQkn0C},
	year={2011},
	publisher={Springer Berlin Heidelberg}
}

@article{AB71,
	title={Variational solutions of some nonlinear free boundary problems},
	author={J. F. G. Auchmuty and Richard Beals},
	journal={Archive for Rational Mechanics and Analysis},
	year={1971},
	volume={43},
	pages={255-271},
	url={https://api.semanticscholar.org/CorpusID:122838332}
}

@article{JM19,
	title = {On rotating axisymmetric solutions of the {E}uler–{P}oisson equations},
	journal = {Journal of Differential Equations},
	volume = {266},
	number = {7},
	pages = {3942-3972},
	year = {2019},
	issn = {0022-0396},
	doi = {https://doi.org/10.1016/j.jde.2018.09.023},
	url = {https://www.sciencedirect.com/science/article/pii/S0022039618305667},
	author = {Juhi Jang and Tetu Makino},
}

@article{JS22,
	place = {Country unknown/Code not available}, 
	title = {On Uniformly Rotating Binary Stars and Galaxies}, 
	url = {https://par.nsf.gov/biblio/10334799}, 
	DOI = {10.1007/s00205-022-01766-4}, 
	abstractNote = {}, 
	journal = {Archive for Rational Mechanics and Analysis}, 
	volume = {244}, 
	number = {2}, 
	author = {Juhi Jang and Jinmyoung Seok}, 
	year = {2022},
}

@article{Li91,
	author={YanYan Li},
	title={On uniformly rotating stars},
	journal={Archive for Rational Mechanics and Analysis},
	year={1991},
	month={Dec},
	day={01},
	volume={115},
	number={4},
	pages={367-393},
	issn={1432-0673},
	doi={10.1007/BF00375280},
	url={https://doi.org/10.1007/BF00375280}
}

@article{LY87,
	author={Lieb, Elliott H.
	and Yau, Horng-Tzer},
	title={The {C}handrasekhar theory of stellar collapse as the limit of quantum mechanics},
	journal={Communications in Mathematical Physics},
	year={1987},
	month={Mar},
	day={01},
	volume={112},
	number={1},
	pages={147-174},
	issn={1432-0916},
	doi={10.1007/BF01217684},
	url={https://doi.org/10.1007/BF01217684}
}

@article{McC06,
	title={STABLE ROTATING BINARY STARS AND FLUID IN A TUBE},
	author={Robert J. McCann},
	year={2006},
	journal={Houston Journal of Mathematics},
	volume={32(2)},
	pages={603–631},
	url={https://api.semanticscholar.org/CorpusID:7123278}
}

\end{document}